\documentclass[cmp,final,envcountsect,envcountsame]{svjour}

\def\makeheadbox{\relax}
\def\commname{\relax}

\usepackage{geometry}
\smartqed

\usepackage{times}

\usepackage{hyperref}

\usepackage{amsmath,amsfonts, amssymb}
\usepackage{thmtools,thm-restate}
\usepackage{xy, xypic}
\usepackage{enumerate}
\usepackage[mathscr]{euscript}
\usepackage{tikz}
\usetikzlibrary{arrows}

\usepackage{extarrows}

\newcommand{\cstar}{$C^*$-al\-ge\-bra}
\newcommand{\cp}{\mathcal{P}}
\newcommand{\ch}{\mathcal{H}}
\newcommand{\ci}{\mathcal{I}}
\newcommand{\cz}{\mathcal{Z}}
\newcommand{\ca}{\mathcal{A}}
\newcommand{\cb}{\mathcal{B}}
\newcommand{\cg}{\mathcal{G}}
\newcommand{\ck}{\mathscr{K}}
\newcommand{\ct}{\mathscr{T}}
\newcommand{\tk}{\tilde{K}}
\newcommand{\tf}{\tilde{F}}
\newcommand{\NN}{\mathbb{N}}
\newcommand{\C}{\mathbb{C}}
\newcommand{\R}{\mathbb{R}}
\newcommand{\Z}{\mathbb{Z}}

\newcommand{\srem}{s^R}
\newcommand{\codiag}{\underrightarrow{Diag}}
\newcommand{\diag}{\underleftarrow{Diag}}
\newcommand{\colim}{\underrightarrow{lim}}
\newcommand{\llim}{\underleftarrow{lim}}
\newcommand{\ad}{\mathrm{Ad}_u}
\newcommand{\Spec}{\mathrm{Spec}}

\newcommand{\op}{\mathsf{op}}

\newcommand{\Exists}[1]{\exists{#1}\boldsymbol{.}\;}

\def\simMvN{\sim_{_\mathbf{M}}}
\def\leqMvN{\preccurlyeq_{_\mathbf{M}}}
\def\geqMvN{\succcurlyeq_{_\mathbf{M}}}
\def\simu{\sim_{_\mathbf{u}}}
\def\lequ{\preccurlyeq_{_\mathbf{u}}}

\def\V#1{V_{#1}}
\def\P#1{\Phi(#1)}
\def\Pf{\Phi}  
\def\Pzf{\Phi_z}  
\def\Pp#1{\Phi_p(#1)}  
\def\Ppf{\Phi_p}  
\def\II{\ci}
\def\TT{\ct}

\def\cat#1{\mathbf{\mathsf{#1}}}
\def\Spec{\mathrm{Spec}}
\def\Prim{\mathrm{Prim}}
\def\Hom{\mathrm{Hom}}

\def\dim{\mathrm{dim}}
\def\dom{\mathrm{dom}}
\def\cod{\mathrm{cod}}
\def\ker{\mathrm{ker}}

\newcommand{\catC}{\cat{C}}
\newcommand{\catA}{\cat{A}}
\newcommand{\catB}{\cat{B}}

\newcommand{\longhookrightarrow}{\ensuremath{\lhook\joinrel\relbar\joinrel\rightarrow}}
\newcommand{\setdef}[2]{\left\{#1 \mid #2\right\}}             
\newcommand{\fdec}    [3]{#1\colon #2 \longrightarrow #3}
\newcommand{\incfdec} [3]{#1\colon #2 \longhookrightarrow #3}

\newcommand{\Mdot}{ \text{ .}}
\newcommand{\Mcomma}{ \text{ ,}}
\newcommand{\Msemicolon}{ \text{ ;}}
\newcommand{\Mand}{\quad\quad \text{ and } \quad\quad}
\newcommand{\Miff}{\quad\quad \text{ iff } \quad\quad}

\usepackage{textcomp}
\newcommand{\Gelfand}{Gel\textquotesingle\-fand}
\newcommand{\Neumark}{Na\u{\i}\-mark}
\newcommand{\Naimark}{Na\u{\i}\-mark}

\newcommand{\SDF}{G}

\newcommand{\lift}[1]{#1}

\begin{document}
\title{Contextuality and Noncommutative Geometry in Quantum Mechanics}
\author{Nadish de Silva\inst{1} \and Rui Soares Barbosa\inst{2} }
\authorrunning{N. de Silva \and R. S. Barbosa}
\institute{Department of Computer Science, University College London,
Gower Street, London WC1E 6BT
\\ \email{nadish.desilva@utoronto.ca}
\and
Department of Computer Science, University of Oxford,
Wolfson Building, Parks Road, Oxford OX1 3QD
\\ \email{rui.soares.barbosa@cs.ox.ac.uk} }
\date{7th June 2018}
%
\communicated{}
\maketitle
\begin{abstract}
Observable properties of a classical physical system can be modelled deterministically as functions from the space of pure states to outcomes; dually, states can be modelled as functions from the algebra of observables to outcomes.  The probabilistic predictions of quantum physics are contextual in that they preclude this classical assumption of reality: noncommuting observables, which are not assumed to be comeasurable, cannot be consistently ascribed deterministic values even if one enriches the description of a quantum state.

Here, we consider the geometrically dual objects of noncommutative operator algebras of observables as being generalisations of classical (deterministic) state spaces to the quantum setting and argue that these generalised geometric spaces represent the objects of study of noncommutative operator geometry.  By adapting the spectral presheaf of Hamilton--Isham--Butterfield, a formulation of quantum state space that collates contextual data, we reconstruct tools of noncommutative geometry in an explicitly geometric fashion.  In this way, we bridge the foundations of quantum mechanics with the foundations of noncommutative geometry \`a la Connes et al.

To each unital $C^*$-algebra $\ca$ we associate a geometric object---a diagram of topological spaces collating quotient spaces of the noncommutative space underlying $\ca$---that performs the role of a generalised \Gelfand\ spectrum. 
We show how any functor $F$ from compact Hausdorff spaces to a suitable target category can be applied directly to these geometric objects to automatically yield an extension $\tilde{F}$ acting on all unital $C^*$-algebras.

This procedure is used to give a novel formulation of the operator $K_0$-functor via a finitary variant $\tilde K_f$ of the extension $\tk$  of the topological $K$-functor.

We then delineate a $C^*$-algebraic conjecture that the extension of the functor that assigns to a topological space its lattice of open sets assigns to a unital \cstar\ the Zariski topological lattice of its primitive ideal spectrum, i.e. its lattice of closed, two-sided ideals.
We prove the von Neumann algebraic analogue of this conjecture.
\end{abstract}

\newpage
\setcounter{tocdepth}{3}
\tableofcontents
\newpage

\section{Introduction}\label{sec:intro}

The mathematical description of classical physical systems
exhibits an elegant interplay between logico-algebraic aspects
of observables (and propositions), reflecting the arithmetic of quantities,
and topologico-geometrical aspects of a space of states.
A system can be described in two (dually) equivalent ways, depending on whether one takes states or observables as primary.
Adopting a \emph{realist} or \emph{ontological} perspective, one starts with a space of states,
and constructs observables as (continuous) functions from states to scalar values.
Conversely, adopting an \emph{operational} or \emph{epistemic} perspective, one starts with an algebra of observables and constructs states as (homomorphic) functions from observables to  scalar values.
Such state-observable dualities are manifestations of the duality between geometry and algebra that is a common thread running throughout mathematics.  We are particularly interested in the interplay described by the \Gelfand--\Naimark\ duality \cite{GelfandNaimark} between the categories of
unital commutative $C^*$-algebras (of observables) and compact Hausdorff spaces (of states).

Quantum systems are described by their $C^*$-algebra of observables which, by the \Gelfand--\Neumark--Segal construction \cite{GelfandNaimark,Segal47}, can be represented as an algebra of Hilbert space operators.
However, as quantum algebras are noncommutative, \Gelfand--\Naimark\ duality cannot be used to obtain a geometric description as in the classical case.
Indeed, pure quantum states do not ascribe deterministic values to all observables; rather, a quantum state yields for each observable a probability distribution on the various outcomes possible upon measurement. 

The inherently probabilistic nature of quantum mechanics has discomfited advocates of \emph{physical realism} since the theory's inception.
Einstein \cite{EPR}, in his famous foundational debates with Bohr, argued that the quantum state does not provide a `complete description' of a system.  These debates led to the study of \emph{hidden variable models} of quantum theory: i.e. ones in which quantum states are represented as probability distributions over a space of  more fundamental \emph{ontic states} that yield deterministic values for all observables.  
Motivated by a desire to hold onto realism, one may insist that a hidden variable model be \emph{noncontextual}: that the values of the system's observable properties be independent of the precise method of observation, and,
in particular, of which other observables are measured simultaneously.
However, the no-go theorem of Bell--Kochen--Specker \cite{bell1966,KochenSpecker}
rules out hidden variable models of this kind,
showing that \emph{contextuality} is a necessary feature of any theory reproducing the highly-verified empirical predictions of quantum mechanics.

The primary motivation of this work is to study a candidate geometric notion of state space for quantum systems that maintains as closely as possible a realist perspective in the sense alluded to above.  In pursuing this, we identify and explore a connection with the well-studied mathematical field of noncommutative geometry: our geometric notion of state space will be the geometric dual of a noncommutative algebra of observables.  Our desired geometric construction must necessarily account for contextuality as an obstacle towards a naively ontological quantum state space.

Our starting point is the \emph{spectral presheaf} formulation of the Bell--Kochen--Specker theorem.
Hamilton, Isham and Butterfield \cite{oldtopos1,oldtopos3} associate to a von Neumann algebra
a presheaf of compact Hausdorff spaces, varying over \emph{contexts} (commutative von Neumann subalgebras representing sets of
comeasurable observables).
The Kochen--Specker theorem finds expression as the nonexistence of a global section of points (i.e. a global point in the generalised `space'),
whereas Gleason's theorem can be expressed as a correspondence between quantum states and global sections of probability distributions (i.e. a global probability distribution on the generalised `space').
These observations strongly suggest the role the spectral presheaf might play
as a notion of quantum state space that fundamentally incorporates contextuality.
Indeed, this idea forms the basis of
a considerable body of research into topos-theoretic approaches to quantum physics
by Isham, D\"{o}ring, et al. \cite{oldtopos1,oldtopos2,oldtopos3,oldtopos4,newtopos1,newtopos2,newtopos3,newtopos4} and by Heunen, Landsman, Spitters, et al. \cite{Dutch,toposdutch-bohrlogic,toposdutch-deepbeauty}.
More recent developments \cite{DoeringHarding:AbelianSubalgebrasAndTheJordanStructureOfAvNAlgebra,Doering2012:GeneralisedGelfandSpectra,Doering2012:FlowsOnGeneralisedGelfandSpectra,Barbosa:DPhil}
pursue the idea of regarding spectral presheaves as providing a generalised notion of space dual to noncommutative von Neumann algebras.

We directly relate this body of research to the programme of noncommutative operator geometry of Connes et al. \cite{connes},
in which mathematicians regard noncommutative operator algebras as generalised geometric spaces.
The result is a plethora of generalisations of geometric tools to the noncommutative algebraic setting that are constructed indirectly via \Gelfand--\Naimark\ duality.

Our contribution is to associate diagrams of topological spaces, akin to the spectral presheaf, to noncommutative algebras and use them to give direct geometric formulations of notions from noncommutative geometry.
We argue this is necessary for any concretely spatial object to be considered a quantum state space in the sense of being dual to a noncommutative algebra.
Physically, the topological spaces in the diagram associated to an algebra can be thought of as state spaces for sets of compatible observables.
Mathematically, they are precisely those quotient spaces of the `noncommutative space' represented by the algebra that are tractable in the sense of being (classical) topological spaces.

The general scheme is as follows:
given a concept defined on (compact Hausdorff) topological spaces (corresponding to unital commutative $C^*$-algebras), one lifts it from the contexts to a global concept by taking a limit,
thus yielding a corresponding extension defined for all unital $C^*$-algebras.
In order to support the connection between the global concepts defined via direct extension and those defined indirectly via \Gelfand--\Naimark\ duality,
we apply this template to extend two different concepts: 
$K$-theory and open sets.
First, we show how a finitary variant of the extension $\tilde{K}$ of the topological $K$-functor yields a novel formulation of the operator $K_0$-functor.
Secondly, we conjecture a correspondence between the extension of the notion of open sets and two-sided ideals of the algebra,
and prove the von Neumann algebraic version of this conjecture.

\subsection*{Summary of main results}  
We define the category $\diag(\catC)$ whose objects are diagrams in the category $\catC$,
i.e. functors from any small category to $\catC$.
We then introduce $\fdec{\llim}{\diag(\catC)}{\catC}$ when $\catC$ is complete.
This generalises the usual limit functors for diagrams of a fixed shape.
These constructions have duals denoted $\codiag(\catC)$ and $\fdec{\colim}{\codiag(\catC)}{\catC}$ when $\catC$ is cocomplete.

The \emph{spatial diagram functor} $\fdec{G}{\cat{uC^*}^\op}{\diag(\cat{KHaus})}$ associates a diagram of compact Hausdorff spaces to each unital $C^*$-algebra:
the objects in the diagram $G(\ca)$ are the \Gelfand\ spectra of unital commutative sub-$C^*$-algebras of $\ca$, while the morphisms 
arise from inner automorphisms of $\ca$.  
We shall also consider some variations (finitary $C^*$-algebraic and von Neumann algebraic) of this construction.

For any functor $\fdec{F}{\cat{KHaus}}{\catC}$ to a complete target category,
we define an \emph{extension} $\fdec{\tilde{F}}{\cat{uC^*}^\op}{\catC}$ that acts on a unital $C^*$-algebra $\ca$ by applying $F$ (lifted to diagrams) to the diagram $G(\ca)$ and then taking the limit:
\[\tilde{F} = \llim \circ F \circ G \; \colon \; \cat{uC^*}^\op \longrightarrow \diag(\cat{KHaus}) \longrightarrow \diag(\catC) \longrightarrow \catC \Mdot\]
The functor $\tilde{F}$ extends $F$.  By this, we mean that the two functors agree on unital commutative $C^*$-algebras:
\[\tf |_{\cat{ucC^*}}  \simeq F \circ \Sigma \Mcomma\]
where the functor $\fdec{\Sigma}{\cat{uC^*}^\op}{\cat{KHaus}}$ maps a unital commutative $C^*$-algebra to its \Gelfand\ spectrum.

We compare the extension of important topological concepts with their existing generalisation in noncommutative geometry.
First, we consider the topological $K$-functor, $\fdec{K}{\cat{KHaus}^\op}{\cat{Ab}}$, and give a novel formulation of operator $K$-theory, $\fdec{K_0}{\cat{C^*}}{\cat{Ab}}$ via a finitary variant $\tilde K_f$ of $\tk$, the extension of the topological $K$-functor:\footnote{The category $\cat{Ab}$ of abelian groups is cocomplete. So, $K$ can be seen as a functor from $\cat{KHaus}$ to a complete category, $\cat{Ab}^\op$. This can be extended as explained earlier to all unital $C^*$-algebras, yielding a functor $\cat{uC^*}\longrightarrow\cat{Ab}$. Note that a limit in the category $\cat{Ab^\op}$ is a colimit in $\cat{Ab}$. This is then extended to the category $\cat{C^*}$ of all (i.e. not-necessarily-unital) \cstar s via unitalisation, in the same fashion that $K_0$ is.}

\begin{restatable*}{theorem}{thmKtheory}
\label{thm:Ktheory}
$K_0 \simeq \tk_{f}$ when restricted to the full subcategory of stabilisations of unital $C^*$-algebras.
Equivalently, $K_0 \simeq K_0 \circ \ck \simeq \tk_{f} \circ \ck$ as functors $\cat{uC^*} \longrightarrow \cat{Ab}$.
Consequently, $\fdec{K_0}{\cat{C^*}}{\cat{Ab}}$ is naturally isomorphic to the extension via unitalisation of the functor $\tk_{f} \circ \ck$.
\end{restatable*} 
In the statement above,
$\ck$ is the stabilisation functor and $\tk_{f}$ is the finitary version of $\tk$,
in the sense that the extension of $K$ is obtained, for a unital $C^*$-algebra $\ca$,
via a diagram $G_f(\ca)$ of the \Gelfand\ spectra of its unital, \emph{finite-dimensional}, commutative sub-$C^*$-algebras.
Since stable $C^*$-algebras are nonunital, this then needs to be extended to all $C^*$-algebras, which is done via unitalisation.

We then consider the functor $\fdec{\TT}{\cat{KHaus}}{\cat{CMSLat}}$
that maps a compact Hausdorff space to its lattice of closed sets ordered by reverse inclusion (which is isomorphic to the lattice of open sets ordered by inclusion)
and a continuous function to its direct image map.\footnote{Dealing with closed sets makes the action on morphisms easier to state, as it is given by the map taking a set to its image, whereas for open sets one would have the map taking a set to the complement of the image of its complement.})
Let $\II$ denote the functor mapping a unital $C^*$-algebra to its lattice of closed, two-sided ideals
(equivalently, the lattice of open sets of the $C^*$-algebra's primitive ideal spectrum)
and a unital $*$-homomorphism to its preimage map.
We conjecture that $\tilde{\TT} \simeq \II$, and prove the von Neumann algebraic analogue:

\newif\iffirststatement
\firststatementtrue
\begin{restatable*}{theorem}{thmVNideals}
\label{thm:VNideals}
Let $\fdec{\TT_\mathsf{W}}{\cat{HStonean}}{\cat{CMSLat}}$ be the functor assigning to a hyperstonean space 
its lattice of clopen sets ordered by reverse inclusion and to an open continuous function its direct image map,
and let $\fdec{\II_\mathsf{W}}{\cat{vNA}^\op}{\cat{CMSLat}}$ be the functor assigning to a von Neumann algebra its lattice of ultraweakly closed, two-sided ideals and to an ultraweakly continuous (or normal) unital $*$-homomorphism its inverse image map.
Then $\tilde{\TT}_\mathsf{W} \simeq \II_\mathsf{W}$, where $\tilde{\TT}_\mathsf{W}$ is the von Neumann algebraic extension of $\TT$.
\end{restatable*}
\firststatementfalse

Here, the extension $\tilde{\TT}_\mathsf{W}$ is obtained, for each von Neumann algebra $\ca$, via a diagram $G_{\mathsf{W}}(\ca)$ whose objects are the spectra of its commutative sub-von Neumann algebras and whose morphisms arise from inner automorphisms of $\ca$.

\subsection*{Outline}
The remainder of this article is organised as follows:
\begin{itemize}
\item 
Section~\ref{sec:background-motivation} surveys the main aspects of state-observable dualities, quantum contextuality, the spectral pre\-sheaf, and noncommutative geometry, and expands on the motivation for this work;
\item 
Section~\ref{sec:diagrams} introduces the necessary technical machinery for functorially associating diagrams of topological spaces to operator algebras;
\item 
Section~\ref{sec:extensions} defines the notion of an \emph{extension} of a concept defined for compact Hausdorff topological spaces to one defined for all unital $C^*$-algebras;  
\item 
Section~\ref{sec:Ktheory} considers the extension of topological $K$-theory and gives a novel geometric formulation of operator $K$-theory;
\item
Section~\ref{sec:ideals} explains the conjectured correspondence between extended open sets and closed, two-sided ideals, and proves the von Neumann algebraic analogue;
\item 
Section~\ref{sec:conclusions} outlines future lines of research.
\end{itemize}

The appendices contain additional and background material:
\begin{itemize}
\item
Appendix~\ref{app:colimits} presents an alternative explicit construction of the colimit functor of Section~\ref{ssec:generalisedcolimit};
\item
Appendix~\ref{app:ktheory} contains background material on topological and operator $K$-theory, expanding on the presentation in Sections~\ref{ssec:topKtheory-technicalbackground} and \ref{ssec:opKtheory-technicalbackground};
\item
Appendix~\ref{app:ideals} contains background material on the primitive ideal spectrum of a $C^*$-algebra and some facts about von Neumann algebras needed in Section~\ref{sec:ideals}.
\end{itemize}

The present article is based on the doctoral dissertation of the first author \cite{deSilva:DPhil}.  Earlier versions of the main results have appeared in the unpublished manuscripts \cite{deSilva:Ktheory} (Sections~\ref{sec:diagrams}--\ref{sec:Ktheory} and Conjecture \ref{conj:idealsCstar}) and \cite{deSilvaBarbosa:PartialAndTotalIdealsVNAlgebra} (Section~\ref{sec:ideals}).

\subsection*{Notation}

For simplicity, given a functor $\fdec{F}{\catA}{\catB}$, we do not distinguish it notationally from the same map regarded as a functor $\catA^\op \longrightarrow \catB^\op$.
The same applies to $\fdec{G}{\catA^\op}{\catB}$ and $\fdec{G}{\catA}{\catB^\op}$,
since we treat $(\catA^\op)^\op$ as equal to $\catA$.

We shall also denote by $\lift{F}$ the lifting of a functor $\fdec{F}{\catA}{\catB}$ to the categories of diagrams introduced in Section~\ref{ssec:catsdiags}, mapping $\catA$-valued to $\catB$-valued diagrams (see the remarks at the end of that section for details).

Given functors $\fdec{F, G}{\catA}{\catB}$, we write $F \simeq G$ to denote that $F$ and $G$ are naturally isomorphic.

For reference, Table~\ref{tab:catnames} lists the categories mentioned throughout this article,
and their duals when applicable (see Section~\ref{ssec:background-classicaldualities}).

\begin{table}
\caption{Glossary of categories and their duals}
\label{tab:catnames} 
\centering
\begin{tabular}{llll}
\hline\noalign{\smallskip}
    \textbf{Notation}                   & \textbf{Objects}      & \textbf{Morphisms}   & \textbf{Dual}           \\
\noalign{\smallskip}\hline\noalign{\smallskip}
$\cat{C^*}$ & \cstar s & $*$-homomorphisms & \\
$\cat{uC^*}$ & unital \cstar s & unital $*$-homomorphisms & \\
$\cat{ucC^*}$ & unital commutative \cstar s & unital $*$-homomorphisms & $\cat{KHaus}$ \\
$\cat{vNA}$ & von Neumann algebras & ultraweakly continuous (or normal) unital $*$-homomorphisms & \\
$\cat{cvNA}$ & commutative von Neumann algebras & ultraweakly continuous (or normal) unital $*$-homomorphisms & $\cat{HStonean}$\\
$\cat{BA}$ & Boolean algebras & Boolean algebra homomorphism & $\cat{Stone}$\\
$\cat{cBA}$ & complete Boolean algebras & complete Boolean algebra homomorphisms & $\cat{Stonean}$\\
$\cat{caBA}$ & complete atomic Boolean algebras & complete Boolean algebra homomorphisms & $\cat{Set}$\\
$\cat{Set}$ & sets & functions & $\cat{caBA}$ \\
$\cat{Top}$ & topological spaces & continuous functions & \\
$\cat{KHaus}$ & compact Hausdorff spaces & continuous functions & $\cat{ucC^*}$ \\
$\cat{Stone}$ & Stone spaces & continuous functions & $\cat{BA}$ \\
$\cat{Stonean}$ & Stonean spaces & open continuous functions & $\cat{cBA}$ \\
$\cat{HStonean}$ & hyperstonean spaces & open continuous functions & $\cat{cvNA}$ \\
$\cat{Ab}$ & abelian groups & group homomorphisms & \\
$\cat{AbMon}$ & abelian monoids & monoid homomorphisms & \\
$\cat{Rng}$ & rings & ring homomorphisms & \\
$\cat{CMSLat}$ & complete lattices & complete meet-semilattice homomorphisms (meet-preserving functions) & \\
$\cat{Cat}$ & small categories & functors & \\
\noalign{\smallskip}\hline
\end{tabular}
\end{table}

\section{Background and motivation}\label{sec:background-motivation}

We survey the main background topics to make the results accessible to both mathematicians and physicists and to expand on the motivation for our work.

\subsection{(Classical) state-observable dualities}\label{ssec:background-classicaldualities}

Observables, being representatives of quantities that vary with state, are generally endowed with algebraic structure capturing the arithmetic of quantities.
States, on the other hand, are endowed with geometric structure: states are close to each other when they share similar physical properties.

Important examples are those classical systems that can be modelled in terms of Poisson geometry \cite{marsden}.  The collection of pure states is in fact a geometric space: a Poisson manifold.  This justifies the use of the terminology \emph{state space}.  Any smooth real-valued map on this manifold can be taken to represent an observable quantity and, taken together, these maps form a commutative algebra with pointwise operations.  In this case, the Poisson bracket provides the additional structure of a Lie algebra.  Hence, we refer to the \emph{algebra of observables}.

In the above example, predictions for the outcomes of experiments are deterministic and observables are explicitly represented as quantity-valued functions on the state space.  However, the fact that a pairing of a state with an observable results in a quantity means that fixing a state yields a quantity-valued function on the collection of observables.  Identifying a state with the function on observables it defines allows realising the state space as a space of functions from the algebra of observables to an algebra of quantities. 

This perspective is common in duality theory.  The simplest example is the Stone-type duality between the categories $\cat{Set}$ of sets and functions and $\cat{caBA}$ of complete, atomic Boolean algebras and complete Boolean algebra homomorphisms \cite{Tarski1935:ZurGrundlegungBA}.
In one direction, it maps functorially a set $S$ to the Boolean algebra $\Hom_\cat{Set}(S, 2)$ of functions to $2 = \{0,1\}$, and a function $\fdec{f}{S}{T}$ to a $\cat{caBA}$-morphism $\fdec{f^*}{\Hom_\cat{Set}(T, 2)}{\Hom_\cat{Set}(S, 2)}$ given by $f^*(g) = g \circ f$.
Similarly, in the opposite direction, one can use the functor $\Hom_\cat{caBA}(-,2)$, where $2$ is the two-element Boolean algebra, to complete the duality of these categories.  This establishes a (dual) equivalence between a category of geometric objects---sets can be seen as trivial geometries with no structure beyond cardinality---and a category of algebraic objects.

A duality of the same form---defined by $\Hom$ functors to a dualising object $2$---exists between the categories $\cat{Stone}$ of Stone spaces and continuous functions and $\cat{BA}$ of Boolean algebras and Boolean algebra homomorphisms \cite{Stone1936,Stone1937-RepThm,Sikorski-BooleanAlgebras}; see \cite{GivantHalmos-IntroBoolean} as an elementary reference and \cite{Johnstone:StoneSpaces} for more general forms of this duality.
The geometric nature of Stone spaces, which are particular kinds of topological spaces, is clearer in this instance.  This example also demonstrates a logical form of duality between semantics and syntax: the algebraic category of Boolean algebras can be seen as the category of propositional theories whereas the geometric category of Stone spaces is the category of corresponding spaces of two-valued models \cite{Tarski1935:ZurGrundlegungBA,AwodeyForssell2013:FirstOrderLogicalDuality}.

A classic example of geometric-algebraic duality, which informs Section~\ref{sec:ideals}, is that between unital commutative rings and affine schemes \cite{hartshorne}.  Given such a ring $R$, one can define a topological space $\Spec R$ called the {prime spectrum} (or just {spectrum})  whose points are the prime ideals of $R$ and whose open sets are indexed by ideals of $R$.  One can then define a sheaf of commutative rings on $\Spec R$ such that the stalk at a prime ideal $p$ is the localisation of $R$ at $p$, turning $\Spec R$ into a locally ringed space.  The locally ringed spaces that arise in this way are called {affine schemes}.  The commutative ring giving rise to an affine scheme can be recovered by taking the ring of global sections of the scheme.  In this way, a geometric dual to the category of unital commutative rings is constructed and geometric tools and reasoning can be brought to bear in subjects which make use of commutative rings, such as number theory.
Many other examples of geometric-algebraic dualities can be found; see \cite{khalkhali}.

The most important example for our purposes is the \Gelfand--\Naimark\ duality between the category $\cat{KHaus}$ of compact Hausdorff spaces and continuous functions and the category $\cat{ucC^*}$ of unital commutative $C^*$-algebras and unital $*$-homomorphisms \cite{GelfandNaimark}.
Under this duality, a space $X$ is mapped to the unital commutative $C^*$-algebra $C(X)$ of all the continuous complex-valued functions on $X$.
The reversal of this process---going from a commutative algebra $\ca$ to the topological space whose algebra of functions is $\ca$---is accomplished by the \Gelfand\ spectrum functor $\Sigma$.
The points of the space $\Sigma(\ca)$ are the characters of $\ca$, i.e. unital homomorphisms from $\ca$ to $\C$,
with topology given by pointwise convergence (the weak-$*$ topology).
So, similarly to the Stone dualities discussed above, \Gelfand--\Naimark\ duality arises from $\Hom$ functors to a dualising object: in this case, $\C$.\footnote{There is also a real version of this duality, with $\R$ as the dualising object \cite{Johnstone:StoneSpaces}.} $\Hom_{\cat{ucC^*}}(-,\C)$ is topologised by pointwise convergence, using the topology of $\C$;  $\Hom_{\cat{Top}}(-,\C)$ inherits (pointwise) the algebraic structure from $\C$ and is given the uniform norm.\footnote{Note that $\C$ is not in fact a compact Hausdorff space, and thus does not live in $\cat{KHaus}$. However, this duality can be extended to one between locally compact Hausdorff spaces and (not- necessarily-unital) $C^*$-algebras.}

\Gelfand--\Naimark\ duality has a clear interpretation as a state-observable duality.  The objects of the geometric category can be seen as state spaces
of classical systems.
 Observables, in this analogy, are the continuous real-valued functions on the state space, i.e. the self-adjoint elements of the algebra of observables.  The \Gelfand\ spectrum functor recovers the pure state space from the algebra of observables.  We attribute a classical nature to these models since states are associated with well-defined values for all observables simultaneously.

Von Neumann algebras constitute an important special class of \cstar s.
The topological spaces that arise as \Gelfand\ spectra of commutative von Neumann algebras are hyperstonean spaces \cite{Dixmier1951:SurCertainsEspaces,Groethendieck1953:ApplicationsLineairesFaiblementCompactesCK}.
These are extremally disconected compact Hausdorff (or Stonean) spaces
with sufficiently many positive normal measures; see \cite[Definition III.14]{Takesaki1979:TheoryOfOperatorAlgebrasI} for more details.
The appropriate notion of morphism when dealing with von Neumann algebras is that of ultraweakly continuous, or normal, unital $*$-homomorphisms. Corresponding to such morphisms between commutative von Neumann algebras are open continuous maps between their spectra. Thus, \Gelfand--\Naimark\ duality restricts to a duality between the categories $\cat{cvNA}$ of commutative von Neumann algebras and ultraweakly continuous, or normal, unital $*$-homomorphisms and $\cat{HStonean}$ of hyperstonean spaces and morphisms of Stonean spaces, i.e. open continuous functions. See e.g. \cite[Chapter III.1]{Takesaki1979:TheoryOfOperatorAlgebrasI} for the objects part of this duality and \cite[Lecture 14]{Lurie2011:261y} for the morphisms.

In all these instances, our algebraic categories consist of objects with commutative operations.  In quantum theory, the model of a system is specified by a noncommutative algebra of observables.  Understanding the geometric duals of these objects is essential to completing our understanding of how quantum mechanics revises the nature of classical  theories and, in particular, notions of states of systems.  It is also a fundamental question of purely mathematical interest.

\subsection{Contextuality: the Bell--Kochen--Specker theorem}
This theorem establishes that quantum mechanics is contextual
in the sense that it does not admit a hidden variable model
where (hidden) ontic states ascribe consistent values to all observables simultaneously,
independent of the method of observation, i.e. of which other observables are measured together with some observable.
In fact, it shows that it is not possible to construct even a single such consistent deterministic ontic state (valuation).

Suppose we have a quantum system whose algebra of observables is $\cb(\ch)$ where $\dim\ch > 2$.
Measurements are given by the self-adjoint operators in the algebra of observables.

\begin{definition}
A valuation on a von Neumann algebra $\ca$
is a map $v$ from the self-adjoint operators of $\ca$ to $\R$ such that
$v(1) = 1$ and
for any pair of commuting observables $A$ and $B$, $v(A+B)=v(A)+v(B)$ and $v(AB)=v(A)v(B)$.
\end{definition}

These conditions are necessary for such a potential hidden state to be consistent with
a quantum state in the sense that it does not predict the occurrence of any impossible events.
Note that when $A$ and $B$ are two commuting observables, then $A+B$ and $AB$ also commute with both $A$ and $B$,
and therefore can be measured together.
Upon performing these measurements on any quantum state,
the obtained joint outcomes satisfy the functional relations above.\footnote{In some presentations (such as \cite{Redhead1987} or, for general von Neumann algebras, \cite[Lemma 6]{doring2005kochen}),
the sum and product rules in the definition of valuations are derived from a different assumption, the functional composition principle (or FUNC principle).
This states the requirement that
$v(f(A)) = f(v(A))$ for a class of functions $\fdec{f}{\R}{\R}$, which in the case of valuations on general von Neumann algebras, is taken to be that of Borel functions.}

\begin{theorem}[Bell--Kochen--Specker \cite{bell1966,KochenSpecker}]
No valuations exist on $\cb(\ch)$ if $\dim\ch > 2$.
\end{theorem}

Observe that, restricted to projections, a valuation is a map that takes the values $0$ or $1$ and is additive on sets of orthogonal projections.
Kochen and Specker proved that such a valuation on projections is impossible to construct by providing a collection $W$ of 117 vectors in a Hilbert space of dimension $3$ such that no subset of $W$ intersects each orthogonal triple in $W$ precisely once.

The result is extended to all separable von Neumann algebras without summands of type $I_1$ or $I_2$
in \cite{doring2005kochen}, showing that no valuations exist
for quantum systems described by algebras of observables of these kinds.

The study of contextuality has enjoyed a revival in recent times.
A number of abstract formalisms to study contextuality
in general physical theories have been developed recently
\cite{AbramskyBrandenburger,AcinEtAl:CombinatorialApproachNonlocalityContextuality,CSW,spek}.
Also, recent work suggests that it might be considered a resource conferring advantage in
computational and information-processing tasks \cite{Raussendorf:contextuality,HowardEtAl:magic}.

\subsection{The spectral presheaf}
The research programme known as the topos approach to quantum theory
aims to achieve a reformulation of quantum theory that resembles the classical picture as closely as possible,
but taking contextuality as a central feature.
The central idea is to study a quantum system via its \emph{contexts} or \emph{classical perspectives}.

Formally, a context may be taken to correspond to a commutative subalgebra of the algebra of observables.  Physically, this represents a set of properties that can be simultaneously measured with one experimental procedure.
The idea to consider quantum systems via classical contexts has a long history, in a sense going back to Bohr \cite{Bohr:DiscussionWithEinsteinOnEpistemologicalProblemsInAtomicPhysics}, and appearing explicitly in Edwards \cite{Edwards1979:TheMathematicalFoundationsOfQM}.

Regarding contextuality as a central aspect of quantum mechanics,
Isham \& Butterfield \cite{oldtopos1,oldtopos2,oldtopos3,oldtopos4}
proposed the use of presheaves to assign data to these contexts and glue it together in a consistent way, 
in order to achieve a full description of a quantum system via the pasting of all its partial classical perspectives.
This idea was further developed by D\"oring \& Isham \cite{newtopos1,newtopos2,newtopos3,newtopos4}
and, along somewhat different lines, by Heunen et al. \cite{Dutch,toposdutch-bohrlogic,toposdutch-deepbeauty}.
The topos approach to quantum mechanics suggests a candidate geometric object
to take the role of the state space in analogy to the classical case: the spectral presheaf.
This object collects the classical partial state spaces of commutative subalgebras along with morphisms used to consistently relate data from different classical perspectives.

\begin{definition}A \emph{context} of a von Neumann algebra $\ca$ is a commutative sub-von Neumann algebra of $\ca$.  The \emph{context category}  $\cat{C}(\ca)$ is the subcategory of  commutative von Neumann algebras whose objects are the contexts of $\ca$ and whose morphisms are the inclusion maps between them. 
\end{definition}

For every context $V$, the \Gelfand\ spectrum functor can be used to construct a sample space $\Sigma(V)$ whose points represent the possible outcomes for a measurement procedure jointly measuring all the observables in $V$.  The elements $o$ of $\Sigma(V)$ are functions assigning real numbers to the observables in $V$ while preserving addition and multiplication.  These conditions are easily justified on physical grounds and are sufficient to guarantee that $o$ assigns to a self-adjoint operator a value on its spectrum.
As explained in Section~\ref{ssec:background-classicaldualities}, this collection of functions comes equipped with an extremally disconnected compact Hausdorff topology (in fact, a hyperstonean topology) coming from pointwise convergence, which is discrete in the case that $\ca$ is finite-dimensional.

\begin{definition}[Spectral presheaf]\label{def:spectralpresheaf}
Let $\ca$ be a von Neumann algebra.
   Its \emph{spectral presheaf} is the functor of type  $\cat{C}(\ca)^\op \longrightarrow \cat{HStonean}$ that maps each object and morphism of $\cat{C}(\ca)$ to its image under the \Gelfand\ spectrum functor.
\end{definition} 

An inclusion map $\iota:\ V \hookrightarrow V'$ corresponds to a coarse-graining,
i.e. the context $V$ represents a procedure measuring a subset of the observables measured by the procedure represented by $V'$.
The image under the \Gelfand\ spectrum functor of such an inclusion, $\fdec{\Sigma(\iota)}{\Sigma(V')}{\Sigma(V)}$, acts by restriction:
an outcome map $o \in \Sigma(V')$ is taken to $o|_{V}$.  

Accordingly, a global section of the spectral presheaf of $\ca$ is a choice of $o_{V} \in \Sigma(V)$ for all contexts $V$ of $\ca$ such that $o_{V}= o_{V'}|_{V}$ whenever $V \subset V'$.
Therefore, the Bell--Kochen--Specker theorem can be reformulated in terms of the spectral presheaf:
\begin{theorem}[\cite{oldtopos3,doring2005kochen}]
Suppose $\ca$ is a separable von Neumann algebra without type $I_1$ or $I_2$ summands
Then its spectral presheaf has no global sections.
\end{theorem}

This is simply a reformulation of the Bell--Kochen--Specker
theorem for these von Neumann algebras \cite{doring2005kochen},
since valuations on a von Neumann algebra $\ca$ correspond to global sections of its spectral presheaf.

Thus, the impossibility of providing a mathematical model in the classical sense for quantum theory
is expressed by constructing a geometric object associated to a quantum system via collating the sample spaces associated to contexts, linked by a simple consistency condition related to coarse-graining, and demonstrating that said object possesses no `global points'.

These geometries represented by spectral presheaves do, however, possess global probability distributions.
Remarkably, these distributions are in correspondence with (possibly mixed) quantum states.
Just as the lack of points of spectral presheaves is equivalent to a landmark theorem of quantum foundations, the Bell--Kochen--Specker theorem,
the correspondence between distributions on spectral presheaves and quantum states is equivalent to Gleason's theorem \cite[see also \cite{Maeda1989:ProbabilityMeasuresOnProjectionsInVonNeumannAlgebras,hamhalter2003quantum}]{gleason1957measures,Christensen1982:MeasuresOnProjectionsAndPhysicalStates,Yeadon1983:MeasuresOnProjectionsInW*algebrasOfTypeII1,Yeadon1984:FinitelyAdditiveMeasuresOnProjectionsInFiniteW*Algebras}.
This observation was first made by de Groote \cite{degroote},
and is succinctly expressed using the framework described in this article (see Section~\ref{ssec:thms-quantum-foundations}).

\subsection{The noncommutative geometry of $C^*$-algebras}

Noncommutative geometry is the mathematical study of noncommutative algebras by the extension of geometric tools that have been rephrased in the language of commutative algebra to the noncommutative setting  \cite{khalkhali}.
Given a duality between geometric objects and commutative algebras, such as \Gelfand--\Naimark\ duality, we can rephrase geometric concepts by expressing them algebraically in terms of functions.
For example, if we wish to algebraically express the idea of an open set of a topological space $X$, we might think about the set of functions that vanish outside of it and note that this is an ideal of $C(X)$.
In fact, there is a bijective correspondence between closed ideals of $C(X)$ and open sets of $X$.
As a more complicated example, the Serre--Swan theorem \cite{swan} allows us to identify vector bundles over $X$ with finitely generated projective $C(X)$-modules.
Remarkably, these algebraic descriptions of geometric concepts do not crucially rely on the commutativity of $C(X)$.
Therefore, one can generalise geometric tools and intuition to noncommutative algebras $\ca$ by using these same algebraic descriptions.
This justifies thinking of a noncommutative \cstar\ as a \textit{noncommutative (locally compact Hausdorff) topological space}.  The elements of the \cstar\ $\ca$ are thought of as continuous complex-valued functions on a metaphorical noncommutative space.  Such a space defies explicit description by conventional mathematical ideas about what a space is; for example, it cannot be thought of as a collection of points for such an object always has a commutative algebra of functions.

One of the best examples of an extension of a topological tool to the setting of noncommutative spaces is that of $K$-theory.  The isomorphism classes of vector bundles over a space $X$ form a semigroup under direct sum and the Grothendieck group of this semigroup is $K(X)$.  The $K$-functor is an important cohomological invariant in the study of topology.  By using the geometry-to-algebra dictionary described above, one defines an extension of $K$ to \cstar s $\ca$ in terms of equivalence classes of finitely generated projective $\ca$-modules; in this way, the operator $K_0$-functor is constructed.  It is an extension of $K$ in the sense that when $\ca$ is commutative, i.e. $\ca \simeq C(X)$ for a space $X$, then $K_0(\ca) \simeq K(X)$.  In this way, we obtain a powerful invariant of \cstar s which forms the basis of a classification programme \cite{elliott1995classification}.  The modern account of operator $K_0$ uses an equivalent formulation in terms of equivalence classes of projections in matrix algebras over $\ca$ \cite{rordam}.

With considerable effort, this process of translation from geometry to algebra yields a conceptual dictionary covering a vast terrain within mathematics.  It is not just topological concepts that can be translated into the language of algebra; there are also noncommutative extensions of measure theory, differential geometry, etc. \cite{connes}; see Table~\ref{tab:dictionary}.

\begin{table}
\caption{Dictionary of concepts between Geometry and Algebra}
\label{tab:dictionary} 
\centering
\begin{tabular}{ll}
\hline\noalign{\smallskip}
    \textbf{Geometry}                   & \textbf{Algebra}                    \\
\noalign{\smallskip}\hline\noalign{\smallskip}
    continuous function from a space to $\C$ & element of the algebra (operator)             \\
    continuous function from a space to $\mathbb{R}$ & self-adjoint element of the algebra \\
                range of a function                                                                     & spectrum of an operator                                                 \\
    open set                            & closed, two-sided ideal               \\
    vector bundle                       & finitely generated projective module  \\
    Cartesian product                   & minimal tensor product              \\
    disjoint union                      & direct sum                          \\
    infinitesimal                       & compact operator                    \\         
    regular Borel probability measure   & state                               \\
    integral                            & trace                               \\
    1-point compactification            & unitalisation                       \\
                \ldots & \ldots \\
\noalign{\smallskip}\hline
\end{tabular}
\end{table}

 \subsection{Motivation}\label{ssec:motivation}

The unreasonable effectiveness of topological tools and intuition in the study of \cstar s suggests the existence of a deeper principle at work.  The method of translating geometric ideas into algebra in order to generalise them is powerful but can be somewhat ad hoc.
Ideally, one may hope for a new conception of space, of which the commutative/topological situation would be a special case, and which would serve as the (objects) of a category dual to that of noncommutative \cstar s.
That is, one would be able to extend the notion of  \Gelfand\ spectrum of a commutative algebra to the noncommutative case by assigning to an algebra $\ca$ such a `space', whose set of continuous functions would be, in some sense, $\ca$.
As pointed out above, an explicit description of (currently imaginary) noncommutative topological spaces is very difficult since such spaces defy most contemporary ideas about mathematical spaces.  It is difficult to know how to begin defining such an object.  However, we can imagine that equipped with such an explicit description, should it not depart too far from the commutative situation, one could find natural and intuitive methods of extending topological tools.


Thus, our criterion for a successful explicit manifestation of noncommutative space is that it naturally leads to extensions of topological concepts that agree with well-known and useful noncommutative geometric concepts.  In effect, we aim to complete the conceptual diagram of Figure~\ref{fig:conceptualdiagram}.
This diagram requires some explanation.  The top row describes the  two dually equivalent mathematical formalisms for encapsulating the operational content of a classical system: the topological picture, in which states are taken as the primitive concept, and the commutative $C^*$-algebraic picture, in which observables are taken as primitive.  

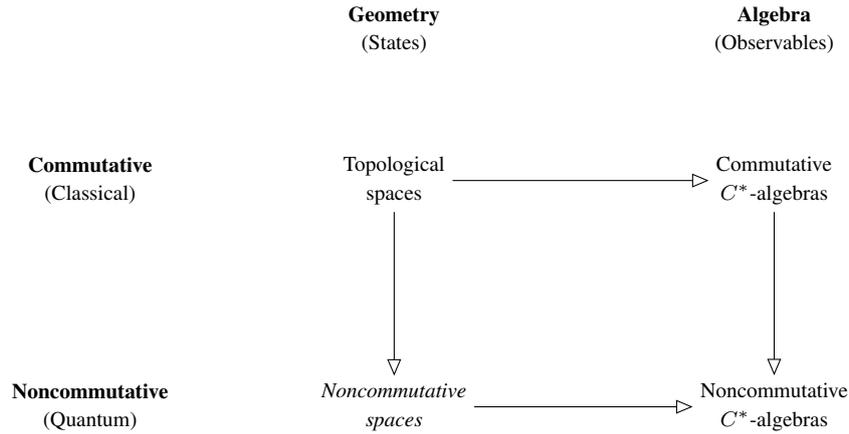
\begin{figure}
\centering
\begin{tikzpicture}[every text node part/.style={align=center}]
\node at (-2,2) {\textbf{Commutative} \\ (Classical)};

\node at (-2,-1) {\textbf{Noncommutative} \\ (Quantum)};

\node at (2,4) {\textbf{Geometry}\\ (States)};

\node at (7,4) {\textbf{Algebra} \\ (Observables)};

\node (v1) at (2,2) {Topological \\ spaces};

\node (v2) at (7,2) {Commutative \\ \cstar s};

\node (v4) at (2,-1) {\emph{Noncommutative} \\ \emph{spaces}};

\node (v3) at (7,-1) {Noncommutative \\ \cstar s};

\draw [-open triangle 45] (v1) edge node[midway,above] {} (v2);
\draw [-open triangle 45] (v2) edge node[midway,right] {} (v3);
\draw [-open triangle 45] (v1) edge node[midway,left] {} (v4);
\draw [-open triangle 45] (v4) edge node[midway, above] {} (v3);
\end{tikzpicture}
\caption{Here, we give a (nonrigorous) high-level diagram representing the heuristic processes by which topological concepts are generalised to the noncommutative setting.  The top and right arrows correspond to the usual method of translation:  the top arrow represents translating a topological notion to an algebraic one via \Gelfand--\Naimark\ duality and the right arrow represents applying this algebraic definition in the noncommutative setting.
Our aim is to generalise the \Gelfand\ spectrum functor $\Sigma$ to a functor $\textbf{G}$ that assigns to each noncommutative algebra a \emph{noncommutative space}.  This requires proposing a candidate construction of noncommutative space that is manifestly geometric.  Further, we ask that this notion of noncommutative space comes naturally equipped with processes corresponding to arrows completing the diagram---a left arrow corresponding to a way of generalising topological concepts to these noncommutatively spatial objects, and a bottom arrow corresponding to porting such (generalised) topological concepts to noncommutative algebraic ones via the new association $\textbf{G}$---in such a way that reproduces the results of the usual translation process.
}
\label{fig:conceptualdiagram}       
\end{figure}

The arrows represent methods for the translation and generalisation of concepts.  The \Gelfand\ spectrum functor allows for any notion or theorem phrased in terms of the topological structure of spaces to be translated into algebraic terms; e.g. open sets of a space becomes closed, two-sided ideals of an algebra.  Once a concept has been phrased in terms of algebra, it can be applied without modification to the noncommutative case; e.g. finitely generated projective modules of a commutative algebra (the equivalent of vector bundles) becomes finitely generated projective modules of a not-necessarily-commutative algebra.  Thus, the composition of the top and right arrows can be seen as the usual process of generating the basic entries of the noncommutative dictionary.  

Note, however, that there is some ambiguity in this translation process.
A topological concept can be translated in several different ways, which means that intuition and judgement must be deployed when determining appropriate algebraic analogues.  As a very simple example, open sets of a space $X$ are in correspondence with both the closed, left ideals of $C(X)$ and the closed, two-sided ideals of $C(X)$ as these two collections are identical in the commutative case.  Thus, finding a completely automatic method of translation that eliminates such ambiguities would in itself constitute an advance in the structural understanding of noncommutative geometry.

Akemann and Pedersen \cite{pedersen} proposed to replace the translation process by working directly with Giles--Kummer's \cite{giles} and Akemann's \cite{akemann} noncommutative generalisations of the basic topological notions of open and closed sets.  In contrast, we do not employ algebraic generalisations of basic topological notions. Instead, we work with objects that slightly generalise the notion of topological space and come readily equipped with an alternative to the translation process.

In addition to the work of Akemann--Pedersen and Giles--Kummer on noncommutative generalisations of \Gelfand--\Naimark\ duality, there have been a number of  alternative approaches by authors including Alfsen \cite{Alfsen}, Bichteler et al. \cite{Bichteler},  Dauns--Hofmann \cite{Dauns}, Fell \cite{Fell}, Heunen et al. \cite{heunen2011gelfand},  Kadison \cite{Kadison}, Kruml et al. \cite{Kruml}, Krusy\'nski--Woronowicz \cite{Kruszynski}, Mulvey \cite{mulvey}, Resende \cite{Resende}, Shultz \cite{Shultz}, and Takesaki \cite{Takesaki:duality}.  An excellent discussion of many of these works is contained in a paper by Fujimoto \cite{Fujimoto}.

Our goal with this work is to find, to a first approximation, a way of completing the conceptual diagram of Figure~\ref{fig:conceptualdiagram}. That is, we aim  to propose and study candidate definitions of a category $\cat{NCSpaces}$ of  \emph{noncommutative spaces} and a generalised \Gelfand\ spectrum functor $\fdec{\textbf{G}}{\cat{uC^*}^\op}{\cat{NCSpaces}}$ acting on the category $\cat{uC^*}$ of \emph{all} unital \cstar s and unital $*$-homomorphisms.
The first motivation is to provide a geometric manifestation for a notion of noncommutative space (namely, the quantum state space described above) whose existence is currently understood as being merely metaphorical.  The second is to exploit this geometric manifestation to obtain a canonical method for importing concepts of topology to noncommutative algebra.

The primary desiderata of a guess for a notion of noncommutative space is that it comes equipped with: i) a natural method of generalising notions from topology (that is, a left arrow in the informal diagram), and ii) a functorial association of such a generalised `space' $\textbf{G}(\ca)$ to each noncommutative algebra $\ca$, which provides a way of translating generalised topological concepts to noncommutative algebras by applying them to the corresponding noncommutative `space' (i.e. a bottom arrow).
That the composition of these two translations match the usual noncommutative dictionary would justify thinking of $\textbf{G}(\ca)$ as the geometric manifestation of a noncommutative algebra $\ca$.

Our proposal for a notion of noncommutative space and a functor $\textbf{G}$, as inspired by Isham and Butterfield's work, is to consider diagrams of topological spaces associated to contexts.
Our primary mathematical objectives are two-fold.
First, to argue that a necessary desideratum of a proposed geometric interpretation of a noncommutative algebra is a precise account of the relationship between topological concepts and their noncommutative analogues.
Indeed, this requirement will be critical for fixing the structure of our diagrams---specifically, the class of morphisms that are included.
Secondly, we aim to provide evidence that this is achievable.

We do not necessarily expect that this will immediately yield a full categorical duality, i.e. a concrete category equivalent to $\cat{uC^*}^\cat{op}$,  but rather stimulate progress towards that goal---or towards a better understanding of the obstacles to that goal. 
Finding such a full duality would require characterising the objects and morphisms of $\cat{NCSpaces}$ that are in the image of a functor $\textbf{G}$ and establishing that that $\textbf{G}$ is faithful and injective on objects---so that its image is a category and $\textbf{G}$ an equivalence onto it.
Note that the particular first approximations to $\textbf{G}$ that we consider in this article---which associate diagrams of topological spaces to a noncommutative algebra---are faithful but not full, and in particular not essentially injective.\footnote{This follows from the existence of a \cstar\ $\ca$ nonisomorphic to its opposite algebra \cite{connes1975factor}.  As both $\ca$ and $\ca^{op}$ have the same commutative subalgebras (contexts), both will be assigned identical diagrams of topological spaces.}
Instances of noncommutative concepts that lack a commutative analogue could provide guidance on which additional data, such as a group action, one might need to take into account when defining a $\textbf{G}$ to achieve a categorical duality.
The failure of a particular guess for $\textbf{G}$ to be essentially injective may also provide such guidance; however, one might also be open to the possibility that nonisomorphic algebras contain equivalent topological information and thus correspond to the same (or homeomorphic) noncommutative space.

Even without a complete categorical duality for \cstar s, the perspective outlined in this article may prove to be useful.  The extent to which noncommutative geometry can be understood directly in topological terms is a wide open---and, in our estimation, interesting---question.

The framework of extensions, developed in Section~\ref{sec:extensions}, formalises how certain ways of associating diagrams of topological spaces to noncommutative algebras come with left and bottom arrows, and in this way yield a noncommutative counterpart for every topological concept.
In Section~\ref{sec:Ktheory}, we determine the appropriate $\textbf{G}$ such that the associated extension of topological $K$-theory essentially matches the established noncommutative $K$-theory. 
In Section~\ref{sec:ideals}, as a verification of this candidate construction of $\textbf{G}$, we conjecture that it can be used to extend the notion of open set to that of closed, two-sided ideal, and prove the analogue of this in the setting of von Neumann algebras.

\section{Spatial diagrams}\label{sec:diagrams}

We introduce the technical machinery necessary for contravariantly functorially associating
diagrams of topological spaces, describing quotient spaces of a `noncommutative space',
to noncommutative \cstar s.
We consider functors that associate to a \cstar\ a diagram whose objects are spectra of contexts and whose morphisms are chosen to yield a natural method, described in the next section, of extending functors that act on compact Hausdorff spaces to functors acting on all unital \cstar s.
An analogous method is applicable to extending functors that act on hyperstonean spaces to functors acting on all von Neumann algebras.

\subsection{The categories of all diagrams in $\catC$}\label{ssec:catsdiags}

We propose to  associate to each unital \cstar\ $\ca$ a diagram of topological spaces whose objects are the spectra of the unital commutative sub-\cstar s of $\ca$.  Given that this association should generalise the \Gelfand\ spectrum functor, we would naturally expect it to be contravariantly functorial.

Typically, one thinks of a diagram $\fdec{D}{\catA}{\catC}$ in a category $\catC$ as living inside the functor category $\catC^\catA$.
This is inadequate for our purposes as different  algebras have different sets of commutative sub-\cstar s and will thus be mapped to diagrams of different shapes.
We introduce a general construction that allows considering diagrams of different shapes on the same footing. 

\begin{definition}For any category $\catC$, $\codiag(\catC)$, the covariant category of all diagrams in $\catC$ has as objects all the functors from any small category to $\catC$;
and the morphisms from a diagram $\fdec{D}{\catA}{\catC}$ to a diagram $\fdec{E}{\catB}{\catC}$
are pairs $(f, \eta)$ where $\fdec{f}{\catA}{\catB}$ is a functor and $\eta$ is a natural transformation from $D$ to $E \circ f$. 

The composition $(g,\mu) \circ (f, \eta)$ of two $\codiag(\catC)$-morphisms
\[\fdec{(f,\eta)}{D_1}{D_2} \Mand \fdec{(g,\mu)}{D_2}{D_3}\]
is given by
$(g \circ f, (\mu_f) \circ \eta)$
where $\mu_f$ is the natural transformation from $D_2 \circ f$ to $D_3 \circ g \circ f$ given by $(\mu_f)_a = \mu_{f(a)}$.

The contravariant category of all diagrams in $\catC$, $\diag(\catC)$, is the category $\codiag(\catC^\op)^\op$.
Its objects are  all contravariant functors from a small category to $\catC$; and the morphisms from a diagram $\fdec{D}{\catA^\op}{\catC}$ to a diagram $\fdec{E}{\catB^\op}{\catC}$ are pairs $(f, \eta)$ where $\fdec{f}{\catB}{\catA}$ is a functor and $\eta$ is a natural transformation from $D \circ f$ to $E$.
\end{definition}

The categories defined above can be constructed by considering the colax-slice and lax-slice 2-categories $\cat{Cat} ~/~ \catC$ \cite{slicenlab} and forgetting the 2-categorical structure.

Note that a functor $\fdec{F}{\catC}{\cat{C'}}$
naturally induces a functor from $\codiag(\catC)$ to $\codiag(\cat{C'})$, which we will also denote by $\lift{F}$.
Explicitly, if $\fdec{D}{\catA}{\catC}$, then $\lift{F}(D)$ is simply $F \circ D$,
while a $\codiag(\catC)$-morphism $(f, \eta)$ is sent to
the $\codiag(\cat{C'})$-morphism $(f, F \eta)$ where $(F \eta)_a = F(\eta_a)$.
In a similar fashion, the functor $F$ also induces a functor $\fdec{F}{\diag(\catC)}{\diag(\cat{C'})}$.
Note that, for contravariant functors, we get the following: a functor $\fdec{G}{\catC^\op}{\cat{C'}}$
induces functors $\fdec{G}{\diag(\catC)^\op}{\codiag(\cat{C'})}$ and $\fdec{G}{\codiag(\catC)^\op}{\diag(\cat{C'})}$,
since $\codiag(\catC^\op)=\diag(\catC)^\op$ and $\diag(\catC^\op)=\codiag(\catC)^\op$.

\subsection{Semispectral functors}

Having defined a category that simultaneously accommodates diagrams of varying shape, we are ready to begin defining our contravariantly functorial associations of diagrams of topological spaces to unital \cstar s.
We will define a class of such  functorial associations.  What all these contravariant functors from the category of unital \cstar s to diagrams of compact Hausdorff spaces have in common is that they will associate to each unital \cstar\ a diagram (i.e. a functor) whose domain is a subcategory of the category of unital commutative \cstar s.  In fact, in each case, the objects of the domain subcategory of the diagram associated to a \cstar\ are its unital commutative sub-\cstar s.  The class of morphisms in this domain subcategory, however, will be allowed to vary.
There is an analogous version for von Neumann algebras where one considers only their commutative sub-von Neumann algebras---this will also be of interest to us.

Our motivating example is the spectral presheaf (see Definition~\ref{def:spectralpresheaf}).
The recipe for its construction, which we aim to generalise, is as follows:\begin{enumerate}
\item  take a  von Neumann algebra $\ca$;
\item consider the subcategory $\cat{C}(\ca)$ of $\cat{cvNA}$ 
whose objects are the commutative sub-von Neumann algebras (contexts) of $\ca$ and whose morphisms are the inclusions between such subalgebras;
\item consider the inclusion functor $i_\ca$ of $\cat{C}(\ca)$ in $\cat{cvNA}$: this is an object of $\codiag(\cat{cvNA})$;
\item compose the (von Neumann) \Gelfand\ spectrum functor 
$\fdec{\Sigma}{\cat{cvNA}^\op}{\cat{HStonean}}$ with this inclusion functor to yield an object of $\diag(\cat{HStonean})$, i.e. a functor $\cat{C}(\ca)^\op \longrightarrow \cat{HStonean}$, called the spectral presheaf of $\ca$.
\end{enumerate}

This association of spectral presheaves to von Neumann algebras can be made functorial in a natural way.
Given an ultra-weakly continuous, or normal, unital $*$-homomorphism $\fdec{\phi}{\ca}{\cb}$,
we can define a $\codiag(\cat{cvNA})$-morphism $(f,\eta)$ as follows:
$\fdec{f}{\cat{C}(\ca)}{\cat{C}(\cb)}$ sends a commutative sub-von Neumann algebra $V$ of $\ca$ to the commutative sub-von Neumann algebra $\phi(V)$ of $\cb$, and an inclusion $V \hookrightarrow V'$ to the inclusion $\phi(V) \hookrightarrow \phi(V')$; while $\fdec{\eta}{i_\ca}{i_\cb \circ f}$ is the natural transformation with components $\eta_V$ defined to be $\fdec{\phi|_V}{V}{\phi(V)}$.
The \Gelfand\ spectrum functor for von Neumann algebras, $\fdec{\Sigma}{\cat{cvNA}^\op}{\cat{HStonean}}$,
lifts to a functor from $\fdec{\Sigma}{\codiag(\cat{cvNA})^\op}{\diag(\cat{HStonean})}$,
mapping $(f,\eta)$ to a $\diag(\cat{HStonean})$-morphism between the spectral presheaves of $\cb$ and of $\ca$.
Overall, this yields a functor of type $\cat{vNA}^\op \longrightarrow \diag(\cat{HStonean})$

We will generalise this recipe to \cstar s. However, we will also want to consider other choices of morphisms to include in our diagrams. In the next section, we see that certain ways of associating diagrams of spaces to algebras automatically yield a method for extending topological functors to functors that act on all unital \cstar s.
The family of morphisms we include in our diagrams determines the resulting method of extensions.  Thus, we will vary the family of morphisms in order to determine the one whose method of extending functors matches up with the canonical generalisation process of noncommutative geometry. This was the motivation behind the reconstruction of the definition of operator $K$-theory.

\begin{definition}\label{def:semispectral}
A functor $\fdec{\sigma}{\cat{uC^*}}{\codiag(\cat{ucC^*})}$ is called \emph{semispectral} if:
\begin{enumerate}
\item For any unital \cstar\ $\ca$, $\sigma(\ca)$ is the inclusion functor of a subcategory $\dom(\sigma(\ca))$
of $\cat{ucC^*}$ whose objects are unital commutative sub-\cstar s of $\ca$;
\item\label{condition:homo} For a unital $*$-homomorphism $\fdec{\phi}{\ca}{\cb}$, $\sigma(\phi)$ is the $\codiag(\cat{ucC^*})$-morphism $\fdec{(f, \eta)}{\sigma(\ca)}{\sigma(\cb)}$, where
$\fdec{f}{\dom(\sigma(\ca))}{\dom(\sigma(\cb))}$ takes a unital commutative sub-\cstar\ $V \subset A$ to the unital commutative sub-\cstar\ $\phi(V) \subset B$, and $\eta$ is the natural transformation with components $\eta_V$ being the unital $*$-homomorphisms $\fdec{\phi|_V}{V}{\phi(V)}$;
\item\label{condition:comm-terminal} If $\ca$ is commutative, then it is terminal in $\dom(\sigma(\ca))$.
\end{enumerate}

Similarly, a functor $\fdec{\sigma}{\cat{vNA}}{\codiag(\cat{cvNA})}$ is called \emph{semispectral} if the analogous
conditions hold, with ``sub-von Neumann algebras'' and ``normal unital $*$-homomorphisms'' substituted as appropriate.
\end{definition}

The third condition will be required below to ensure agreement in the commutative case between a functor and its extension.

\subsection{Spatial diagrams}\label{ssec:spatialdiagrams}
Our primary objects of study will be
spatial diagrams, which are ways of associating diagrams of topological spaces to unital \cstar s (or von Neumann algebras), determined by a semispectral functor. 
Given a semispectral functor, the corresponding spatial diagram functor is obtained via \Gelfand--\Naimark\ duality:

\begin{definition}\label{def:sdf}
Given a semispectral functor $\fdec{\sigma}{\cat{uC^*}}{\codiag(\cat{ucC^*})}$,
its corresponding \emph{spatial diagram functor} $\fdec{\SDF_\sigma}{\cat{uC^*}^\op}{\diag(\cat{KHaus})}$
is given as
\[\SDF_\sigma = \lift{\Sigma} \circ \sigma \; \colon \; \cat{uC^*} \longrightarrow \codiag(\cat{ucC^*}) \longrightarrow \diag(\cat{KHaus})^\op \Mcomma\]
where $\lift{\Sigma}$ is the \Gelfand\ spectrum functor $\cat{uC^*}^\op \longrightarrow \cat{KHaus}$ lifted to diagrams.

Analogously, given a (von Neumann) semispectral functor $\fdec{\sigma}{\cat{vNA}}{\codiag(\cat{cvNA})}$,
its corresponding spatial diagram functor $\fdec{\SDF_\sigma}{\cat{vNA}^\op}{\diag(\cat{HStonean})}$
is given as
\[\SDF_\sigma = \lift{\Sigma} \circ \sigma \; \colon \; \cat{vNA} \longrightarrow \codiag(\cat{cvNA}) \longrightarrow \diag(\cat{HStonean})^\op \Mcomma\]
where $\lift{\Sigma}$ is the (von Neumann) \Gelfand\ spectrum functor $\cat{cvNA}^\op \longrightarrow \cat{HStonean}$ lifted to diagrams.
\end{definition}

As explained in the previous section, the first example
(for von Neumann algebras)
is given by the spectral presheaf functor
which is obtained from the semispectral functor 
that sends a von Neumann algebra to the diagram consisting of its von Neumann subalgebras and inclusions between them.
An analogous definition of spectral presheaf can also be given for \cstar s.

For our main results, however, we will need to consider other semispectral functors (and corresponding spatial diagram functors),
which also take into account unitary equivalences between subalgebras. We now give these definitions.

\begin{definition}
Given a unital \cstar\ $\ca$, its \emph{unitary subcategory} $\cat{S}(\ca)$ of $\cat{ucC^*}$
has as objects the unital commutative sub-\cstar s of $\ca$, and as morphisms the restrictions of inner automorphisms of $\ca$.
That is, the morphisms between two unital commutative sub-\cstar s $V, V' \subset \ca$ are precisely those $\fdec{\ad|_V^{V'}}{V}{V'}$ of the form $\ad|_V^{V'}(v) = uvu^*$ for some unitary $u \in \ca$ such that $uVu^* \subset V'$.
\end{definition}

The composition of two such morphisms is given as conjugation by the product of their respective unitaries, which is also a unitary, and so $\cat{S}(\ca)$ is indeed a subcategory of $\cat{ucC^*}$.
Note that any morphism $\fdec{\ad|_V^{V'}}{V}{V'}$ can be decomposed as  $i \circ r$ where $r$ is the isomorphism ${\ad|_V^{uVu^*}}$ between $V$ and $uVu^*$ defined by conjugation by $u$ and where $i$ is the inclusion $uVu^* \hookrightarrow V'$.

\begin{definition}\label{def:unitarysemispectral-Cstar}
The \emph{unitary semispectral functor} $\fdec{g}{\cat{uC^*}}{\codiag(\cat{ucC^*})}$ sends a unital \cstar\ $\ca$ to the inclusion functor $\fdec{\iota_\ca}{\cat{S}(\ca)}{\cat{ucC^*}}$.
The action of $g$ on unital $*$-morphisms is fixed by Condition~\ref{condition:homo} in Definition~\ref{def:semispectral}: given a unital $*$-homomorphism $\fdec{\phi}{\ca}{\cb}$, its image $g(\phi)$ is $(f, \eta)$ where $\fdec{f}{\cat{S}(\ca)}{\cat{S}(\cb)}$ is the functor taking a unital commutative sub-\cstar\ $V \subset \ca$ to $\phi(V) \subset \cb$ and $\eta$ is the natural transformation whose component at $V$ is the unital $*$-homomorphism $\fdec{\phi|_V}{V}{\phi(V)}$.

We denote by $G$ the corresponding spatial diagram functor $\fdec{\SDF_g = \Sigma \circ g}{\cat{uC^*}^\op}{\diag(\cat{KHaus})}$.
\end{definition}

Note that when $\ca$ is commutative, the morphisms in $\cat{S}(\ca)$ are simply the inclusions, which is why Condition~\ref{condition:comm-terminal} of Definition~\ref{def:semispectral} holds.

The topological spaces in the diagram $G(\ca)$ should be thought of as being those that arise as quotient spaces of the hypothetical noncommutative space underlying $\ca$.
To see this, note that a sub-\cstar\ $V$ of $C(X)$ yields an inclusion $\fdec{i}{V}{C(X)}$ which corresponds to a continuous surjection $\fdec{\Sigma(i)}{X}{\Sigma(V)}$. This surjection is a quotient map since both spaces are compact and Hausdorff \cite[p. 12]{tomdieck2008algebraic}.
Thus, in accordance with the central tenet of noncommutative geometry, unital sub-\cstar s of a unital noncommutative algebra $\ca$ are to be understood as having an underlying noncommutative space that is a quotient space of the noncommutative space underlying $\ca$.
By considering only the commutative subalgebras, we are restricting our attention to the tractable quotient spaces: those that are genuine topological spaces.  The morphisms of the diagram serve to track how these quotient spaces fit together inside the larger noncommutative space.

We will require in our analysis of operator $K_0$ a slight modification of the unitary subcategory:

\begin{definition}\label{def:finitaryunitarysubcategory-CStar}
Given a unital \cstar\ $\ca$, its \emph{finitary unitary subcategory} $\cat{S}_f(\ca)$ of $\cat{ucC^*}$
has as objects the unital, \emph{finite-dimensional} commutative sub-\cstar s of $\ca$, and as morphisms the restrictions of inner automorphisms of $\ca$.
\end{definition}

This is used to define a functor $g_f$, which is a finitary version of the unitary semispectral functor $g$:

\begin{definition}\label{def:finitaryunitarysemispectral-CStar}
The \emph{finitary unitary semispectral functor} $\fdec{g_f}{\cat{uC^*}}{\codiag(\cat{ucC^*})}$ sends a unital \cstar\ $\ca$ to the inclusion functor $\fdec{\iota_\ca}{\cat{S}_f(\ca)}{\cat{ucC^*}}$.
For a unital $*$-homomorphism $\fdec{\phi}{\ca}{\cb}$, its image $g_f(\phi)$ is defined to be $(f, \eta)$ where $\fdec{f}{\cat{S}_f(\ca)}{\cat{S}_f(\cb)}$ is the functor taking a unital, finite-dimensional, commutative sub-\cstar\ $V \subset \ca$ to $\phi(V) \subset \cb$ and $\eta$ is the natural transformation whose component at $V$ is the unital $*$-homomorphism $\fdec{\phi|_V}{V}{\phi(V)}$.

We denote by $G_f$ the corresponding spatial diagram functor $\fdec{\SDF_{g_f} = \Sigma \circ g_f}{\cat{uC^*}^\op}{\diag(\cat{KHaus})}$.
\end{definition}

For the main result of Section~\ref{sec:ideals}, we will be dealing with von Neumann algebras only, and as such we require an analogous version of the spatial diagram functor $G$ in this setting:
\begin{definition}\label{def:unitarysemispectral-vN}
Given a von Neumann algebra $\ca$, its \emph{unitary subcategory}  $\cat{S}_\mathsf{W}(\ca)$ of $\cat{cvNA}$
has as objects the commutative sub-von Neumann algebras of $\ca$ and as morphisms the restrictions of inner automorphisms of $\ca$.

The \emph{(von Neumann) unitary semispectral functor} $\fdec{g_\mathsf{W}}{\cat{vNA}}{\codiag(\cat{cvNA})}$ sends a von Neumann algebra $\ca$ to the inclusion functor $\fdec{\iota_\ca}{\cat{S}_\mathsf{W}(\ca)}{\cat{cvNA}}$ and is defined on a normal unital $*$-homomorphism $\fdec{f}{\ca}{\cb}$ in a manner analogous to Definition~\ref{def:unitarysemispectral-Cstar}.

We denote by $G_\mathsf{W}$ the corresponding spatial diagram functor $\fdec{\SDF_{g_\mathsf{W}} = \Sigma \circ g_\mathsf{W}}{\cat{vNA}^\op}{\diag(\cat{HStonean})}$.
\end{definition}

\section{Extensions of topological functors}\label{sec:extensions}

We give a generalisation of limit and colimit functors that act on certain functor categories to ones that act on categories of diagrams. 
This allows us to define the extension of a topological functor to a noncommutative algebraic one, given a semispectral functor as described in the previous section. 
The extension process is interpreted as decomposing a noncommutative space into tractable quotient spaces, applying a topological functor to each one, and pasting the results together.
We illustrate this construction by presenting formulations of (generalised versions of)
the Bell--Kochen--Specker and Gleason's theorems in this framework.

\subsection{The generalised limit and colimit functors}\label{ssec:generalisedcolimit}

When a category $\catC$ is cocomplete,
it admits a colimit functor $\fdec{\colim}{\catC^\catA}{\catC}$
for diagrams over any fixed shape $\catA$. 

A key  feature of the construction of $\codiag(\catC)$ in the case where $\catC$ is cocomplete is the existence of a generalised colimit functor $\fdec{\colim}{\codiag(\catC)}{\catC}$.
It assigns to a diagram $\fdec{D}{\catA}{\catC}$ the same $\catC$-object that is assigned to $D$ by the usual colimit functor for $\catA$-shaped diagrams, of type $\catC^\catA \longrightarrow \catC$.
If $\eta$ is a natural transformation between $D$ and a diagram $\fdec{D'}{\catA}{\catC}$ of the same shape, i.e. a $\catC^\catA$-morphism, then the generalised $\colim$ functor maps the $\codiag(\catC)$-morphism $(id_\catA, \eta)$ between $D$ and $D'$ to the same $\catC$-morphism that is assigned to $\eta$ by the usual colimit functor $\catC^\catA \longrightarrow \catC$.
What is novel is the ability to assign $\catC$-morphisms between colimits of diagrams of different shapes to $\codiag(\catC)$-morphisms between these diagrams.

In this section, we give a concise description of the generalised colimit functor; in Appendix~\ref{app:colimits}, we present a more direct and explicit construction in terms of coequalisers and coproducts.  Everything in this section applies equally well---that is, all dual statements hold true---when $\catC$ is complete, in which case we have a generalised limit functor $\fdec{\llim}{\diag(\catC)}{\catC}$.

First, note that a $\codiag(\catC)$-morphism $(f,\eta)$ between diagrams $\fdec{D}{\catA}{\catC}$ and $\fdec{E}{\catB}{\catC}$ can be decomposed as
\begin{equation}\label{eq:decompose-feta}
(f,\eta) = (f, id_{E \circ f}) \circ (id_\catA, \eta) \; \colon \; D \longrightarrow E \circ f \longrightarrow E
\end{equation}
where $id_{E \circ f}$ is the identity natural transformation from the diagram $\fdec{E \circ f}{\catA}{\catC}$ to itself.

We shall show in detail
how to define the (functorial) action of $\fdec{\colim}{\codiag(\catC)}{\catC}$ on morphisms of the form 
$\fdec{(f, id_{E \circ f})}{E \circ f}{E}$ for $\fdec{f}{\catA}{\catB}$ and $\fdec{E}{\catB}{\catC}$. 
The action on morphisms of the form $(id_\catA, \eta)$ for $\eta$ a natural transformation between diagrams $\fdec{D, E}{\catA}{\catC}$
simply reduces to that of the usual colimit functor $\fdec{\colim}{\catC^\catA}{\catC}$ on the $\catC^\catA$-morphism $\eta$.
The action for a general morphism $(f,\eta)$ is then determined by decomposition \eqref{eq:decompose-feta} and the need to obey functoriality.

Recall that a cocone of a diagram $\fdec{E}{\catB}{\catC}$ is a pair $(K,k)$ consisting of a $\catC$-object $K$
together with a natural transformation $k$ from $E$ to the constant diagram to the fixed object $K$.  In other words, a cocone is an association to each $\catB$-object $b$ of a $\catC$-morphism $\fdec{k_b}{E(b)}{K}$
such that for every $\catB$-morphism $\fdec{h}{b}{b'}$ we have that $k_b = k_{b'} \circ E(h)$, i.e. the following diagram commutes:
\[
\xymatrix{
E(b) \ar[rr]^{E(h)} \ar[rd]_{k_b} &  & E(b')  \ar[ld]^{k_{b'}} \\
& K & 
}
\]
A cocone $(L, l)$ of $E$ is a colimit of $E$ if, for any other cocone $(K, k)$ of $E$, there is a unique $\catC$-morphism $\fdec{m}{L}{K}$ satisfying $k_b = m \circ l_b$ for every $\catB$-objects $b$. The situation is summarised in the following diagram:
\[\xymatrix{
E(b)  \ar@/_/[ddddrr]_{k_b} \ar[ddrr]_{l_b} \ar[rrrr]^{E(h)} &&&& E(b') \ar@/^/[ddddll]^{k_{b'}} \ar[ddll]^{l_{b'}} \\
&&&&\\
& &  \ar@{.>}^m[dd] L &&\\
&&&&\\
& &   K  & &
 }
\]
When a colimit of a diagram exists (as is the case for $E$ since we are assuming $\catC$ to be cocomplete), it is unique up to isomorphism. We can therefore speak of $\textit{the}$ colimit of $E$ and write $\colim(E)$ for the object $L$ above.

Given a $\codiag(\catC)$-morphism of the form $\fdec{(f, id_{E \circ f})}{E \circ f}{E}$, we must define a $\catC$-morphism from $\colim(E \circ f)$ to $\colim(E)$.
The fact that $(\colim(E), l)$ is a cocone for $E$ implies that $(\colim(E), l_f)$ is a cocone for $E \circ f$, where for an object $a$ of $\catA$, $(l_f)_a = l_{f(a)}$.
The definition of colimit for the diagram $E \circ f$
provides a uniquely determined  $\catC$-morphism $\fdec{m}{\colim(E \circ f)}{\colim(E)}$ 
that maps the cocone associated to the colimit of $E \circ f$ to the cocone $(\colim(E), l_f)$.
We can therefore define $\colim( f, id_{E \circ f} )$ to be this morphism $m$.
The universal property is then used to show functoriality of $\colim$ on the class of morphisms of the form $(f,id_{E \circ f})$.

\def\eqj#1{=\;\;\;\;&\;\;\;\;\;\;\;\;\{\text{{\small #1}}\}}

If we additionally have a natural transformation $\fdec{\mu}{E}{E'}$,
by a similar argument, all cocones of $E'$ yield cocones of $E \circ f$ and thus universality implies that there is a unique morphism
from $\colim(E \circ f)$ to $\colim(E')$. We thus conclude that
\begin{equation}\label{eq:connect-idmu-fid}
\colim(id_B,\mu) \circ \colim(f, id_{E\circ f}) = \colim(f, id_{E' \circ f}) \circ \colim(id_\catA,\mu_f) \Mcomma
\end{equation}
where $\colim$ is only being applied to morphisms of each of the two classes from decomposition \eqref{eq:decompose-feta}, for which it has already been defined.
Together with functoriality of the colimit on each of these two classes, this suffices to demonstrate that the entire action of the colimit is functorial:
if for diagrams $\fdec{D_i}{\catA_i}{\catC}$ with $i \in \{1,2,3\}$
we have morphisms $\fdec{(f, \eta)}{D_1}{D_2}$ and $\fdec{(g, \mu)}{D_2}{D_3}$,
then 
\begin{align*}
& \colim(g,\mu) \circ \colim(f,\eta)
\\
\eqj{by definition of $\colim$ on a general morphism}
\\ 
& \colim(g,id_{D_3 \circ g}) \circ \colim(id_{\catA_2},\mu) \circ \colim(f,id_{D_2 \circ f}) \circ \colim(id_{\catA_1},\eta)
\\
\eqj{by \eqref{eq:connect-idmu-fid}}
\\
& \colim(g,id_{D_3 \circ g}) \circ \colim(f, id_{D_3 \circ g \circ f}) \circ \colim(id_{\catA_1},\mu_f) \circ \colim(id_{\catA_1},\eta)
\\
\eqj{by functoriality of $\colim$ on each class}
\\
& \colim(g \circ f, id_{D_3 \circ g \circ f}) \circ \colim(id_{\catA_1},\mu_f \circ \eta)
\\
\eqj{by definition of $\colim$ on a general morphism}
\\
& \colim(g\circ f, \mu_f \circ \eta)
\\
\eqj{by definition of composition in $\codiag(\catC)$}
\\
& \colim((g,\mu) \circ (f,\eta)).
\end{align*}

We give an explicit description of this generalised colimit construction for the case of diagrams of abelian groups, which will be needed in Section~\ref{sec:Ktheory}.
This is computed as the instantiation
of the construction of colimits in terms of coequalisers and coproducts given in Appendix~\ref{app:colimits}
to the category $\cat{Ab}$ of abelian groups and group homomorphisms.

Let $\fdec{D}{\catA}{\cat{Ab}}$ and $\fdec{E}{\catB}{\cat{Ab}}$ be two diagrams of abelian groups and $(f, \eta)$ be a  $\codiag(\cat{Ab})$-morphism from $D$ to $E$.  

First, we describe the colimit of $D$ in $\cat{Ab}$, and thus its image under the functor $\fdec{\colim}{\codiag(\cat{Ab})}{\cat{Ab}}$.  Consider the direct sum of the groups $D(a)$ over all objects $a$ in $\catA$.
If $g$ is an element of the group $D(a)$, we use the notation $(g)_a$ to indicate the element of this direct sum with $g$ in the $a$-th component and 0 in all the others.
The colimit of $D$ is this direct sum modulo the identifications along the morphisms in the diagram $D$;
more precisely, modulo the subgroup generated by the elements $(g)_{a} - (D(h)(g))_{a'}$ where $g \in D(a)$ and $\fdec{h}{a}{a'}$ is an $\catA$-morphism.


We now describe the image of $(f, \eta)$ under the functor $\fdec{\colim}{\codiag(\cat{Ab})}{\cat{Ab}}$. 
It is enough to indicate how the group homomorphism $\colim(f, \eta)$ acts on elements of the colimit of $D$ of the form $[(g)_a]$.
It does so by mapping $[(g)_a]$ to the element $[(\eta_a(g))_{f(a)}]$ of the colimit of $E$.
This is well-defined for if an $\catA$-morphism $\fdec{h}{a}{a'}$ identifies, over $D$, $(g)_{a}$ with $(D(h)(g))_{a'}$,
then the $\catB$-morphism $\fdec{f h}{f(a)}{f(a')}$ identifies, over $E$,
$(\eta_a(g))_{f(a)}$ with $(E(f h) (\eta_a(g)))_{f(a')}$, which is equal to
$(\eta_{a'}(D(h)(g)))_{f(a')}$ by naturality of $\eta$, i.e. by $(E \circ f)(h) \circ \eta_a = \eta_{a'} \circ D(h)$.

\subsection{Extensions of functors}\label{ssec:extensionsfunctors}

For a fixed semispectral functor $\sigma$, we define a natural method for extending functors $\fdec{F}{\cat{KHaus}}{\catC}$ when $\catC$ is complete.
The idea is to use $\sigma$ to turn an algebra $\ca$ into a diagram of commutative algebras, apply the \Gelfand\ spectrum functor contextwise to this diagram to yield a diagram of topological spaces, apply $F$ contextwise to yield a diagram in $\catC$, and finally, apply the extended limit functor $\fdec{\llim}{\diag(\catC)}{\catC}$.  This procedure is also described for the von Neumann algebraic case.

Intuitively, one should think of the extension process as decomposing a noncommutative space into its quotient spaces, retaining those which are genuine topological spaces, applying the topological functor to each one of them, and pasting together the results.  Varying the semispectral functor effectively changes the precise method of gluing together the topological data into a single $\catC$-object.

\begin{definition}\label{def:sigma-extension}
Given a semispectral functor $\fdec{\sigma}{\cat{uC^*}}{\codiag(\cat{ucC^*})}$, a complete category $\catC$, and a functor $\fdec{F}{\cat{KHaus}}{\catC}$, the \emph{${\sigma}$-extension of F}, denoted $\fdec{\tf_\sigma}{\cat{uC^*}^\op}{\catC}$, is given by
\begin{align*}
\tf_\sigma \; =& \; \llim \circ \lift{F} \circ \SDF_\sigma
\\ \; \colon& \;
\cat{uC^*}^\op \longrightarrow \diag(\cat{KHaus}) \longrightarrow \diag(\catC) \longrightarrow \catC  \Mdot
\end{align*}
\end{definition}

Note that $F$ in the right-hand side of the above expression stands for the functor from
$\diag(\cat{KHaus})$ to  $\diag(\catC)$ induced by the given $\fdec{F}{\cat{KHaus}}{\catC}$ (see the last paragraph of Section~\ref{ssec:catsdiags}),
while $\SDF_\sigma$ is the spatial diagram functor corresponding to $\sigma$ (Definition~\ref{def:sdf}).

Extensions are analogously defined with respect to von Neumann algebraic semispectral functors $\fdec{\sigma}{\cat{vNA}}{\codiag(\cat{cvNA})}$, with a functor $\fdec{F}{\cat{HStonean}}{\catC}$ to a complete category $\catC$ yielding an extension $\fdec{\tf_\sigma}{\cat{vNA^*}^\op}{\catC}$.

In some applications
(notably in the case of the topological $K$-functor that we shall consider in the next section),
we are interested in extending a contravariant functor $F$ from $\cat{KHaus}$ to a cocomplete category $\catC$.
This is naturally covered by the definition above by applying it to $\catC^\op$ as the target category.
The process yields an extension $\fdec{\tf_\sigma}{\cat{uC^*}}{\catC}$,
whose last step---taking a limit in $\catC^\op$---corresponds to taking a colimit in $\catC$.
Explicitly, in this instance, we have
\begin{align*}
\tf_\sigma \; =& \; \llim_{\catC^\op} \circ \lift{F} \circ \SDF_\sigma
\\ \; \colon& \;
\cat{uC^*}^\op \longrightarrow \diag(\cat{KHaus}) \longrightarrow \diag(\catC^\op) \longrightarrow \catC^\op  \Mcomma
\end{align*}
but we could also write
\begin{align*}
\tf_\sigma \; =& \; \colim_\catC \circ \lift{F} \circ \SDF_\sigma
\\ \; \colon& \;
\cat{uC^*} \longrightarrow \diag(\cat{KHaus})^\op \longrightarrow \codiag(\catC) \longrightarrow \catC
\Mcomma
\end{align*}
with $\lift{F}$ in this expression standing for the contravariant functor from $\diag(\cat{KHaus})$ to $\codiag(\catC)$ induced by $\fdec{F}{\cat{KHaus}^\op}{\catC}$ (see the last paragraph of Section~\ref{ssec:catsdiags}).

%

The third property in the definition of semispectral functor---that the category picked out by $\sigma(\ca)$ has $\ca$ as a terminal object when $\ca$ is commutative---is crucial in ensuring that $\tf_\sigma$ does indeed extend $F$.
As a consequence of this condition, the diagram $(\lift{F} \circ \lift{\Sigma} \circ \sigma)(\ca)$ has $F(\Sigma(\ca))$ as an initial object, which is therefore equal to its limit (up to isomorphism).
Hence, we have that $\tf_\sigma(\ca) \simeq (F \circ \Sigma)(\ca)$ for every commutative $\ca$.
The second property then ensures that given a homomorphism $\fdec{\phi}{\ca}{\cb}$ between commutative algebras, $\tf_\sigma(\phi)$ completes the commutative square formed by these isomorphisms and $(F \circ \Sigma)(\phi)$, i.e.
\[
\xymatrix{
\tf_\sigma(\ca) \ar[dd]^{\tf_\sigma(\phi)} \ar[rr]^{\simeq} &  & (F \circ \Sigma)(\ca) \ar[dd]^{(F \circ \Sigma)(\phi)}
\\  \\
\tf_\sigma(\cb) \ar[rr]^{\simeq} &  &  (F \circ \Sigma)(\cb)
}
\]
commutes. Thus, these isomorphisms define a natural equivalence between $\tf_\sigma|_{\cat{uCC^*}^\op}$ and $F \circ \Sigma$.  We have thus proved that:

\begin{theorem}
For a semispectral functor $\sigma$, a complete category $\catC$, and a functor $\fdec{F}{\cat{KHaus}}{\catC}$, 
$\tf_\sigma|_{\cat{ucC^*}^\op}  \simeq F \circ \Sigma$.
\end{theorem}

We are primarily interested the case that $\sigma$ is the unitary semispectral functor $g$ as in Definition~\ref{def:unitarysemispectral-Cstar}:
we shall reserve the notation $\tf$ to denote the extension $\tf_g$ of $F$ relative to this functor.
Similarly, we will write $\tf_f$ for $\llim \circ \lift{F} \circ \lift{\Sigma} \circ g_f$,
the finitary restriction of the extension of the contravariant functor $F$, using the finitary version of the unitary spectral functor given in Definition~\ref{def:finitaryunitarysemispectral-CStar}.
Similarly, we write $\tf_\mathsf{W}$ for the extension of a functor $\fdec{F}{\cat{HStonean}}{\catC}$ relative to the
unitary semispectral functor for von Neumann algebras, $g_\mathsf{W}$ from Definition~\ref{def:unitarysemispectral-vN}.

\subsection{Theorems of quantum foundations}\label{ssec:thms-quantum-foundations}

Having established the framework of extensions, we demonstrate how they can be used to succinctly express two fundamental theorems of quantum foundations:
the Bell--Kochen--Specker theorem \cite{bell1966,KochenSpecker} and Gleason's theorem \cite{gleason1957measures}.

The first of these reformulations is due to Hamilton, Isham, and Butterfield  \cite{oldtopos3}, here similarly stated for the generalised version of the Bell--Kochen--Specker theorem by D\"oring \cite{doring2005kochen}.
The second was given by de Groote \cite{degroote},
and it is a reformulation of the generalised version of Gleason's theorem to most von Neumann algebras,
due to Christensen \cite{Christensen1982:MeasuresOnProjectionsAndPhysicalStates}, Yeadon \cite{Yeadon1983:MeasuresOnProjectionsInW*algebrasOfTypeII1,Yeadon1984:FinitelyAdditiveMeasuresOnProjectionsInFiniteW*Algebras}, and others
(see \cite{Maeda1989:ProbabilityMeasuresOnProjectionsInVonNeumannAlgebras,hamhalter2003quantum}).

In this section, we consider the spectral presheaf functor
$\fdec{\Sigma \circ \sigma}{\cat{vNA}^\op}{\diag(\cat{HStonean})}$ obtained from a semispectral functor
$\fdec{\sigma}{\cat{vNA}}{\codiag(\cat{cvNA})}$
as described at the start of Section~\ref{ssec:spatialdiagrams}.
We write $\tf_\sigma = \llim \circ \lift{F} \circ \lift{\Sigma} \circ \sigma$ for the $\sigma$-extension of a functor $\fdec{F}{\cat{HStonean}}{\catC}$ whose target $\catC$ is a complete category.
We also restrict our extensions to the full subcategory of those von Neumann algebras that are separable and contain no type $I_2$ direct summands. In the statement of the Bell--Kochen--Specker, we also have to exclude abelian (type $I_1$) summands.


\begin{theorem}[Reformulation of the Bell--Kochen--Specker theorem \cite{oldtopos3,doring2005kochen}]The extension $\tilde{U}_\sigma$ of the forgetful functor $\fdec{U}{\cat{HStonean}}{\cat{Set}}$ yields the empty set on separable von Neumann algebras without type $I_1$ or $I_2$ summands.
\end{theorem}


This can be interpreted as saying that the notion of `points' cannot be extended (in our sense) from the commutative to the noncommutative world, or that a `noncommutative space' is not a geometry in the usual sense: a collection of `points' with some additional structure.

Let $\fdec{D}{\cat{HStonean}}{\cat{Set}}$ be the regular Borel probability distribution functor which assigns to a hyperstonean topological space $X$ the set of all regular\footnote{A Borel measure $\mu$ is said to be \emph{regular} if, for any Borel set $X$,
\[\mu(X) = \inf\setdef{\mu(U)}{X \subset U, \text{$U$ open}} = \sup\setdef{\mu(K)}{K\subset X, \text{$K$ compact}} \Mdot \]
Note that for compact spaces, these coincide with the (in general, weaker) notion of Radon measure \cite[Corollary 7.6]{Folland:RealAnalysis}. By the Riesz--Markov--Kakutani representation theorem \cite{Riesz1909,Markov1938:MeanValuesExteriorDensities,Kakutani1941:ConcreteRepresentation} (see e.g. \cite[Theorem 7.2]{Folland:RealAnalysis}), regular Borel measures on a compact Hausdorff space $X$ are in one-to-one correspondence with positive linear functionals of unit norm, i.e. states, of the commutative \cstar\ $C(X)$.} Borel probability measures on $X$ and to a continuous function $f$ the corresponding pushforward map $f_*$ on measures, defined by $f_*\mu (e) = \mu(f^{-1}(e))$.

\begin{theorem}[Reformulation of the Gleason--Christensen--Yeadon theorem \cite{degroote}]The extension $\tilde{D}_\sigma$ of the regular Borel probability distribution functor is naturally isomorphic, on the subcategory of separable von Neumann algebras without type $I_2$ summands, to the contravariant functor mapping such a von Neumann algebra to its set of states (positive linear functionals of unit norm) and a normal unital $*$-homomorphisms $\fdec{\phi}{\ca}{\cb}$ to the corresponding pullback that takes a state $\rho$ of $\cb$ to the state $\rho \circ \phi$ of $\ca$.
\end{theorem}

These two theorems can be read as indicating that while the `noncommutative space' $\Sigma \circ \sigma$ has no points, it nonetheless
admits globally consistent probability distributions, and that these correspond to quantum states.

\section{Reconstructing operator $K$-theory}\label{sec:Ktheory}


Topological $K$-theory, defined in terms of vector bundles, is an extraordinary cohomology theory.  Its $C^*$-algebraic generalisation, operator $K$-theory, is similarly defined in terms of the noncommutative analogue of vector bundles, i.e. finitely generated projective modules, and plays an important role in the study of \cstar s, e.g. as a classifying invariant \cite{elliott1995classification}.

%

In this section, we consider the extension of the topological $K$-theory functor, $\fdec{K}{\cat{KHaus}^\op}{\cat{Ab}}$.
The most natural conjecture is that this extension yields operator $K$-theory on unital \cstar s, $\fdec{K_0}{\cat{uC^*}}{\cat{Ab}}$:\footnote{This would be sufficient to recover operator $K$-theory on arbitrary \cstar s, since the functor $\fdec{K_0}{\cat{C^*}}{\cat{Ab}}$ is obtained from its unital version via unitalisation, as we shall see.}

\begin{conjecture}[\cite{deSilva:Ktheory}]
$\fdec{K_0 \simeq \tk}{\cat{uC^*}}{\cat{Ab}}$.
\end{conjecture}


As detailed in Section~\ref{ssec:opKtheory-technicalbackground}, operator $K$-theory is stable.
That is,  we have $K_0 \simeq K_0 \circ \ck$, where $\ck$ is the stabilisation functor (see Definition~\ref{def:stab-functor} below).
Therefore, one could think of operator $K$-theory as being defined only for the stabilisations of unital \cstar s.
The $K$-theory of other unital \cstar s can be obtained by first stabilising the algebra and then applying the $K_0$-functor restricted to this class of algebras; moreover, the $K$-theory of arbitrary \cstar s can then be obtained as usual via unitalisation.



This justifies weakening the conjecture to require only that $K_0$ and $\tk$
agree on the full subcategory of non-unital algebras arising as stabilisations of unital \cstar s.
Note that this necessitates extending $\tk$ to nonunital \cstar s, which is done via unitalisation,
following the same procedure used to extend $K_0$ from unital to all \cstar s.
Equivalently, this weakened conjecture would
require that $K_0 \simeq K_0 \circ \ck \simeq \tk \circ \ck$ as functors $\cat{uC^*} \longrightarrow \cat{Ab}$
(and consequently, that the extensions via unitalisation of $K_0$ and $\tk \circ \ck$ be naturally isomorphic functors of type $\cat{C^*} \longrightarrow \cat{Ab}$).

In fact, we encounter the need to further modify this conjecture by limiting our spatial diagrams of stable \cstar s to include only the finite quotient spaces. The main result proven in this section is thus the following.

\thmKtheory

Here, $\tk_f$ is defined for unital \cstar s as the extension of $K$ relative to the finitary version of the unitary semispectral functor given in Definition~\ref{def:finitaryunitarysemispectral-CStar}; see also Definition~\ref{def:sigma-extension} of the general extension process and the subsequent remarks regarding contravariant functors.
Explicitly, 
\begin{align*}
\tk_f \; =& \;  \colim \circ \lift{K} \circ G_f 
\\
\; \colon& \;  \cat{uC^*} \longrightarrow \diag(\cat{KHaus})^\op \longrightarrow \codiag(\cat{Ab}) \longrightarrow \cat{Ab}
\Mdot
\end{align*}
Note that as stable \cstar s are nonunital, we need to extend $\tk_f$ to nonunital \cstar s using unitalisation.

We thus find that operator $K$-theory, $K_0$, can be defined as a colimit of topological vector bundles over \emph{finite} quotient spaces of a noncommutative space.  This result suggests fixing the appropriate class of morphisms in our ansatz semispectral functor, i.e. the diagrams associated to \cstar s, to be the restrictions of inner automorphisms.

\subsection{Technical background: Topological $K$-theory}\label{ssec:topKtheory-technicalbackground}
We introduce the basic background on topological and operator $K$-theory,
focusing on the minimum required to follow the proof of Theorem~\ref{thm:Ktheory}.
A slightly more detailed presentation can be found in Appendix~\ref{app:ktheory},
or see e.g. \cite{rordam,wegge1993k,fillmore}.

\begin{definition}
For a compact Hausdorff space $X$, the \emph{vector bundle monoid} $V(X)$ is the set of isomorphism classes of complex vector bundles over $X$ with the abelian addition operation of fibrewise direct sum: $[E] + [F] = [E \oplus F]$.  A continuous function $\fdec{f}{X}{Y}$  yields a monoid homomorphism $\fdec{V(f)}{V(Y)}{V(X)}$ by the pullback of bundles, $V(f)([E]) = [f^*E]$.  This defines a functor $\fdec{V}{\cat{KHaus}^\op}{\cat{AbMon}}$, where $\cat{AbMon}$ stands for the category of abelian monoids and monoid homomorphisms.
\end{definition} 

\begin{definition}For an abelian monoid $M$, the \emph{Grothendieck group} of $M$,  $\cg(M)$, is the abelian group $(M \times M)\slash{\sim}$ where $\sim$ is the equivalence relation given by
\[(a, b) \sim (c, d) \Miff \Exists{e \in M} a + d + e =b +c +e \Mdot\]
For a monoid homomorphism $\fdec{\phi}{M}{N}$, the group homomorphism $\fdec{\cg(\phi)}{\cg(M)}{\cg(N)}$ is given by  $\cg(\phi)([(a, b)])  = [(\phi(a),\phi(b))]$.  This defines a functor $\fdec{\cg}{\cat{AbMon}}{\cat{Ab}}$.
\end{definition}

Intuitively, an element $[(a,b)]$ of $\cg(M)$ can be thought of as a formal difference $a-b$ of elements of $M$.
With this interpretation in mind, it is easy to see that $\cg(M)$ is indeed a group, with addition given componentwise, neutral element $[(0,0)]$, and the inverse of $[(a,b)]$ equal to $[(b,a)]$. Moreover, there is a monoid homomorphism $\fdec{i}{M}{\cg(M)}$ given by $a \longmapsto [(a,0)]$.
As an example, the Grothendieck group of the additive monoid of natural numbers (including zero) is the additive group of the integers.

The Groethendieck group functor $\cg$ is an explicit presentation of the group completion functor,
the left adjoint to the forgetful functor from $\cat{Ab}$ to $\cat{AbMon}$.
This means that  $\cg(M)$ is the `most general' group containing a homomorphic image of $M$, in the sense that it satisfies the universal property that 
any monoid homomorphism from $M$ to an abelian group factors uniquely through the homomorphism $\fdec{i}{M}{\cg(M)}$.

\begin{definition}
The \emph{topological $K$-functor} $\fdec{K}{\cat{KHaus}^\op}{\cat{Ab}}$ is $\cg \circ V$.
\end{definition}

\subsection{Technical background: Operator $K$-theory}\label{ssec:opKtheory-technicalbackground}
Following the usual method of noncommutative geometry, in order to generalise a topological concept to the noncommutative case, one must begin with a characterisation of the topological concept in question in terms of commutative algebra.
In the case of $K$-theory, this requires phrasing the notion of a \emph{complex vector bundle over $X$} in terms of the algebra $C(X)$ of continuous, complex-valued functions on $X$:

\begin{theorem}[Serre--Swan, \cite{swan}]
The category of complex vector bundles over a compact Hausdorff space $X$ is equivalent to the category of finitely generated projective $C(X)$-modules.
\end{theorem}

Finitely generated projective $\ca$-modules can be identified with (equivalence classes of) projections in matrix algebras $M_n(\ca)$ over the \cstar\ $\ca$, which are more convenient to work with. 
We are now ready to define operator $K_0$ for unital \cstar s.

\begin{definition} Let $\ca$ be a \cstar. Two projections $p \in M_n(\ca)$ and $q \in M_m(\ca)$ are Murray--von Neumann equivalent, denoted $p \simMvN q$, whenever there is a partial isometry $v \in M_{m,n}(\ca)$ such that $p = vv^*$ and $q = v^*v$.
\end{definition}

\begin{definition}[The Murray--von Neumann semigroup for unital $\ca$]
Let $\ca$ be a unital \cstar.
Its \emph{Murray--von Neumann semigroup}, $V_0(\ca)$, is the set of Murray--von Neumann equivalence classes of projections in matrices over $\ca$:
\[ \left.\bigsqcup_{n \in \NN} \setdef{p \in M_n(\ca)}{p \text{ is a projection}} \middle\slash {\simMvN}\right. \Mdot \]
It is equipped with the abelian addition operation
\[[p] + [q] = \left[\left( \begin{array}{ccc}
p & 0 \\
0 & q \\
\end{array} \right) \right]\Mcomma\]
for which the equivalence class of the zero projection is a neutral element, therefore forming an abelian monoid.
A unital $*$-homomorphism $\fdec{\phi}{\ca}{\cb}$ yields a monoid homomorphism
$\fdec{V_0(\phi)}{V_0(\ca)}{V_0(\cb)}$ given by
$[p] \longmapsto [M_n(\phi)(p)]$
for each $n \in \NN$ and $p$ a projection in $M_n(\ca)$,
where $M_n(\phi)$ acts on elements of $M_n(\ca)$ by entrywise application of $\phi$.
This defines a functor $\fdec{V_0}{\cat{uC^*}}{\cat{AbMon}}$.
\end{definition}

\begin{definition}
The \emph{operator $K_{0}$-functor for unital \cstar s}, $\fdec{K_{0}}{\cat{uC^*}}{\cat{Ab}}$,
is $\cg \circ V_0$.
\end{definition}

We now describe the extension of $K_0$ to all \cstar s. The same recipe will later be used to extend other functors from unital to all \cstar s.

Let $\ca$ be a \cstar\ (which may be unital or nonunital).
By minimally adjoining a unit element to $\ca$, one obtains the unitalisation $\ca^+$ (see Definition~\ref{def:minunitalisation}) and a short exact sequence
\[0 \longrightarrow \ca \xlongrightarrow{\iota} \ca^+ \xlongrightarrow{\pi} \C \longrightarrow 0 \Mdot\]
Moreover, $(-)^+$ is a functor from $\cat{C^*}$ to $\cat{uC^*}$.

\begin{definition}\label{def:unital-to-not}
The $K_0$ group of a \cstar\ $\ca$ (unital or not) is the subgroup of $K_0(\ca^+)$ given by the kernel of $K_0(\pi)$.
A $*$-homomorphism $\fdec{\phi}{\ca}{\cb}$ yields a homomorphism from $\ker K_0(\ca^+ \xlongrightarrow{\pi} \C)$ to $\ker K_0(\cb^+ \xlongrightarrow{\pi} \C)$ by restriction of $K_0(\phi^+)$ to the kernel of $K_0(\ca^+\ \xlongrightarrow{\pi} \C)$.
This defines the \emph{operator $K_0$-functor}, $\fdec{K_{0}}{\cat{C^*}}{\cat{Ab}}$.
\end{definition}

We now consider stability, a key property of the operator $K_0$-functor.

\begin{definition}\label{def:stab-functor}
The \emph{stabilisation functor} $\fdec{\ck}{\cat{C^*}}{\cat{C^*}}$ maps a \cstar\ $\ca $ to the \cstar\ $\ca \otimes \ck$ where $\ck$ is the \cstar\ of compact operators on a separable infinite-dimensional Hilbert space.
A $*$-homomorphism $\fdec{\phi}{\ca}{\cb}$ is mapped to $\phi \otimes id_\ck$.
\end{definition}

Since the \cstar\ $\ck \otimes \ck$ is isomorphic to $\ck$,
the stabilisation functor is an idempotent operation, i.e. $\ck \circ \ck \simeq \ck$.

\begin{definition}
A \cstar\ $\ca$ is called \emph{stable} if $\ca$ is isomorphic to its stabilisation $\ck(\ca) = \ca \otimes \ck$.
\end{definition} 

\begin{theorem}
Operator $K$-theory is stable.  That is, $K_0 \simeq K_0 \circ \ck$.
\end{theorem}
 
Consequently, the operator $K_0$-functor is determined by its restriction to stable \cstar s.

The Murray--von Neumann semigroup, and thus the $K_0$-group,
of a unital \cstar\ $\ca$
can be expressed in a rather simple fashion in terms of projections of its stabilisation~\cite[Exercise 6.6]{rordam}.
We require this definition in the proof of Theorem~\ref{thm:Ktheory} and thus describe it in explicit detail.

\begin{definition}\label{def:uequiv}
Let $\ca$ be a \cstar.
Two projections $p$ and $q$ in a \cstar\ $\ca$ are \emph{unitarily equivalent}, denoted by $p \simu q$, whenever there is a unitary $u \in \ca^+$ such that $p = uqu^*$.
We write $[p]$ for the unitary equivalence class of $p$.
\end{definition}

Given projections $p_1, \ldots, p_k \in \ca \otimes \ck$,
one can find pairwise orthogonal representatives of their unitary equivalence classes,
i.e. there exist projections $q_1, \ldots q_k \in \ca \otimes \ck$ such that $p_i \simu q_i$ ($i \in \{1, \ldots,n\})$ and all the $q_i$ are pairwise orthogonal \cite[Exercise 6.6]{rordam}. 

The Murray--von Neumann semigroup for unital \cstar s admits the following alternative characterisation:
\begin{definition}[The Murray--von Neumann semigroup  for unital $\ca$, alternative definition]\label{def:MvNsemigroup-alternative}
Let $\ca$ be a unital \cstar.
The elements of $V_0(\ca)$ are the unitary equivalence classes of projections in $\ck(\ca)$.
The abelian addition operation is given by orthogonal addition. That is, if $p$ and $p'$ are two  projections in $\ca \otimes \ck$, then
$[p] + [p'] = [q + q']$
where $q$ and $q'$ are orthogonal representatives of $[p]$ and $[p']$, respectively (i.e. $p \simu q$, $p' \simu q'$, and $q \perp q'$).
The equivalence class of the zero projection is a neutral element for this operation, making $V_0(\ca)$ an abelian monoid.
A unital $*$-homomorphism $\fdec{\phi}{\ca}{\cb}$ yields a monoid homomorphism by  $\fdec{V_0(\phi)}{V_0(\ca)}{V_0(\cb)}$ by $[p] \longmapsto [\ck(\phi)(p)]$.
This defines a functor $\fdec{V_0}{\cat{uC^*}}{\cat{AbMon}}$.
\end{definition}

Through this reformulation of the Murray--von Neumann semigroup functor $V_0$, we automatically get a new description of $K_0$. Recall that this is obtained by composition with the Grothendieck group functor, as $K_0 = \cg \circ V_0$.  Then, $K_0(\ca)$ is simply the collection of formal differences $$[p] - [q]$$ of elements of $V_0(\ca)$ with $$[p] - [q] = [p'] - [q']$$ precisely when there exists $[r]$ such that $$[p] + [q'] + [r] = [p'] + [q] + [r]\Mdot$$

Composing the action on morphisms of the Grothendieck group functor after the action of $V_0$ just defined, we find that a unital $*$-homomorphism $\fdec{\phi}{\ca}{\cb}$ between unital \cstar s yields an abelian group homomorphism between the $K_0$ groups of $\ca$ and $\cb$ given by 
$$[p] - [q] \longmapsto [\phi(p)] - [\phi(q)] \Mdot$$

\subsection{Main theorem}

We now prove the main theorem of this section (Theorem~\ref{thm:Ktheory}): that the functors $K_0: \cat{uC^*} \longrightarrow \cat{Ab}$ and $\tk_f \circ \ck : \cat{uC^*} \longrightarrow \cat{Ab}$ are naturally isomorphic.  It follows that the $K_0: \cat{C^*} \longrightarrow \cat{Ab}$ functor can be reconstructed in terms of $\tk_f$, the stabilisation functor, and the unitalisation functor.

Recall that the functor $\tk_f$ is defined on the category of unital \cstar s as:
\begin{align*}\tk_f : \cat{uC^*} \longrightarrow \cat{Ab} \; &= \;  \colim \circ \lift{K} \circ G_f \\
&= \;  \colim \circ \lift{K} \circ \Sigma \circ g_f \Mdot
\end{align*}
We will prove the main theorem in three steps:  
\begin{enumerate}
\item Give a simple presentation of the $\tk_f$ group of a unital \cstar\ $\ca$ in terms of its unitary equivalence classes of projections (Lemma~\ref{lemma:kfpresentation}).
\item
Extend the domain category of $\tk_f$ from $\cat{uC^*}$ to the category $\cat{C^*}$ of (not-necessarily-unital) \cstar s and $*$-homomorphisms via unitalisation in the same way that $K_0$ is extended from unital to all \cstar s, and give a similar presentation of the $\tk_f$ group of a (not-necessarily-unital) \cstar.  This is necessary to make sense of the composition $\tk_f \circ \ck$ as all stable \cstar s are nonunital.
\item  Construct a natural isomorphism between $K_0: \cat{uC^*} \longrightarrow \cat{Ab}$ and $\tk_f \circ \ck: \cat{uC^*} \longrightarrow \cat{Ab}$ (Theorem~\ref{thm:Ktheory}).
\end{enumerate}

\begin{lemma}\label{lemma:kfpresentation}
For a unital \cstar\ $\ca$,
\[\tk_f(\ca) \;=\; \langle \;  [p]_u \;|\; [p]_u = [p_1]_u + [p_2]_u \text{ whenever } p = p_1 + p_2 \; \rangle\]
is the group generated by the unitary equivalence classes of projections in $\ca$ modulo the relations coming from addition of orthogonal projections.
Moreover, for a unital $*$-homomorphism  $\fdec{\phi}{\ca}{\cb}$ between unital \cstar s $\ca$ and $\cb$,
\[\tk_f(\phi)([p]_u) =  [\phi(p)]_u \Mdot\]
\end{lemma}  
\begin{proof}

We will first compute the action of $\tk_f$ on objects before computing its action on unital $*$-homomorphisms.  Recall Definitions~\ref{def:finitaryunitarysubcategory-CStar} and \ref{def:finitaryunitarysemispectral-CStar}.
The objects of the finitary unitary subcategory $\cat{S}_f(\ca)$ are the unital, finite-dimen\-sional, commutative sub-\cstar s of $\ca$.  The morphisms are given by the restrictions of inner automorphisms.
These morphisms are all of the form $i \circ r$ where $i$ is an inclusion and $r$ is an isomorphism between subalgebras given by conjugation by a unitary of $\ca$.

Under the \Gelfand\ spectrum functor, the image of such an object $V$ is a finite, discrete space $\Sigma(V)$ whose points are in correspondence with the atomic projections of the subalgebra $V$.  The images of the inclusions $\fdec{i}{V}{V'}$ are surjections $\fdec{\Sigma(i)}{\Sigma(V')}{\Sigma(V)}$ with the property that whenever a point $s \in \Sigma(V)$ corresponds to a projection $p$ atomic in $V$, then $p$ is the sum of the atomic projections in $V'$ that correspond to the points of $(\Sigma(i))^{-1}(s)$.  In turn, an isomorphism $r$, arising from conjugation by a unitary $u$, is sent by $\Sigma$ to a bijection that connects points whose corresponding projections are related by conjugation by $u$.  

Under the topological $K$-functor, each object $\Sigma(V)$ of the diagram of $G_f(\ca)$ yields a direct sum of copies of $\Z$, one for each point. That is, one gets a trivial vector bundle of every possible dimension (and formal inverses) over each point.  Taking the colimit of the diagram $K \circ G_f(\ca) = K \circ \Sigma \circ g_f(\ca)$ then yields, as described at the end of Section~\ref{ssec:generalisedcolimit}, a direct sum of the abelian groups $ K \circ \Sigma(V)$ indexed by the objects $V$ of $\cat{S}_f(\ca)$ modulo the relations generated by the morphisms of $\cat{S}_f(\ca)$.  In our case, this is a quotient of the direct sum of copies of $\Z$, one for each pair $(V, p)$ where $V$ is a finite-dimensional, unital commutative sub-\cstar\ of $\ca$ and $p$ is an atomic projection in $V$.

The images under $K \circ \Sigma$ of the inclusions result in identifying the generator of the copy of $\Z$ associated to a pair $(V, p)$ with the sum of generators associated to pairs $\{(V', p_i)\}_{i}$ whenever $V \subset V'$ and $\sum_i p_i = p$.
Every nonzero projection $p \in \ca$ is an atomic projection in the subalgebra $\C p + \C (1 - p)$, which is included in every subalgebra that contains $p$.  As the generators associated to the same projection $p$ atomic in different subalgebras are all identified in the colimit, we see that we may speak of the element of the colimit group $[(p)]$ associated to $p$ without reference to which subalgebra it appears in.  Thus, the abelian group $\tilde{K}_f(\ca)$ can be viewed as a quotient of the free abelian group generated by the elements $[(p)]$.  The isomorphisms in the diagram ensure that the elements associated to unitarily equivalent projections are identified.  The second class of identifications consists of those between elements $[p]_u$ and $\sum_i[p_i]_u$ whenever the $p_i$ (are mutually orthogonal and) sum to $p$.

We now consider the action of $\tk_f = \colim \circ \lift{K} \circ \Sigma \circ g_f$ on a unital $*$-homomorphism  $\fdec{\phi}{\ca}{\cb}$.  By Definition~\ref{def:finitaryunitarysemispectral-CStar}, $g_f(\phi)$ is defined to be $(f, \eta)$ where $\fdec{f}{\cat{S}_f(\ca)}{\cat{S}_f(\cb)}$ is the functor taking an object $V \subset \ca$ to $\phi(V) \subset \cb$ and $\eta$ is the natural transformation whose component at $V$ is the unital $*$-homomorphism $\fdec{\phi|_V}{V}{\phi(V)}$.

Suppose $[p]_u \in \tk_f(\ca)$ with $p$ a projection in $\ca$.  The element  $[p]_u$ of the colimit can be identified with a trivial vector bundle $B_p$ of dimension one over the point corresponding to $p$ in the space associated by $ \Sigma \circ g_f(\ca)$ to the subalgebra $V_p = \C p + \C (1 - p)$ of $\ca$.  The natural transformation $\eta$ of the morphism of diagrams $g_f(\phi)$ includes a component $\fdec{\eta_{V_p} = \phi|_{V_p}}{V_p}{V_{\phi(p)}}$ that maps $p$ to $\phi(p)$ and $1-p$ to $1 - \phi(p)$.

Under the image of the lifting of $\Sigma$ to diagrams,
this component becomes $\fdec{\Sigma(\phi|_{V_p})}{\Sigma(V_{\phi(p)})}{\Sigma(V_p)}$ that maps the point corresponding to $\phi(p)$ to the one corresponding to $p$ (and the point corresponding to $1-\phi(p)$ to the one corresponding to $1-p$).

Then, under the topological $K$-functor, this becomes
$\fdec{K \circ \Sigma(V_{\phi(p)})}{K \circ \Sigma(V_p)}{K \circ \Sigma(V_{\phi(p)})}$, which pulls back vector bundles along the map $\Sigma(\phi_{V_p})$ and the bundle $B_p$ is pulled back to the trivial vector bundle of dimension one over the point corresponding to $\phi(p)$ in the discrete space $\Sigma(V_{\phi(p)})$.
The pulled back bundle is identified with $[\phi(p)]_u$ in the colimit $\tk_f(\ca)$ and we conclude that $\tk_f(\phi)([p]_u) =  [\phi(p)]_u$.
\qed\end{proof}

We now extend the functor $\tk_f$ to all \cstar s via unitalisation, in the same way as $K_0$ is extended.


\begin{definition}
The $\tk_f$ group of a \cstar\ $\ca$ (unital or not) is the subgroup of $\tk_f(\ca^+)$ given by the kernel of $\tk_f(\pi)$, where $\fdec{\pi}{\ca^+}{\C}$ is the projection map in the unitalisation short exact sequence (c.f. Definition~\ref{def:unital-to-not}).
A $*$-homomorphism $\fdec{\phi}{\ca}{\cb}$ yields a homomorphism from $\ker \tk_f(\ca^+ \xlongrightarrow{\pi} \C)$ to $\ker \tk_f(\cb^+ \xlongrightarrow{\pi} \C)$ by restriction of $\tk_f(\phi^+)$ to the kernel of $\tk_f(\ca^+\ \xlongrightarrow{\pi} \C)$.
This defines the \emph{$\tk_f$ functor}, $\fdec{\tk_f}{\cat{C^*}}{\cat{Ab}}$.
\end{definition}

To check that this extended map is well-defined on morphisms, we must show
that for any $*$-homomorphism $\fdec{\phi}{\ca}{\cb}$, the unital $*$-homomorphism $\tk_f(\phi^+): \tk_f(\ca^+) \longrightarrow \tk_f(\cb^+)$ carries $\ker \tk_f(\ca^+ \xlongrightarrow{\pi_\ca} \C)$ into $\ker \tk_f(\cb^+ \xlongrightarrow{\pi_\cb} \C)$.
Functoriality then follows immediately from that of
$\tk_f \circ (-)^+ \colon \cat{C^*} \longrightarrow \cat{uC^*} \longrightarrow \cat{Ab}$.
This is done by noting that the following diagram in $\cat{uC^*}$ commutes:
 \begin{equation*}
   \hbox{
   \xymatrix{{\ca^+}  \ar[r]^-{\pi} \ar[d]_{\phi^+} & {\C} \\
{\cb^+} \ar[ru]_-{\pi} & {}
   }} \end{equation*}
and, therefore, so does its image under $\fdec{\tk_f}{\cat{uC^*}}{\cat{Ab}}$:
 \begin{equation*}
   \hbox{
   \xymatrix{{\tk_f(\ca^+)}  \ar[r]^-{\tk_f(\pi)} \ar[d]_{\tk_f(\phi^+)} & {\Z} \\
{\tk_f(\cb^+)} \ar[ru]_-{\tk_f(\pi)} & {}
   }} \end{equation*}

A particular consequence of the following lemma is that the new functor $\fdec{\tk_f}{\cat{C^*}}{\cat{Ab}}$ is an extension of original $\fdec{\tk_f}{\cat{uC^*}}{\cat{Ab}}$, i.e. they agree on unital \cstar s.
This justifies not distinguishing notationally between them.

\begin{lemma}\label{lemma:kfpresentation-nonunital}
For a (not-necessarily-unital) \cstar\ $\ca$,
\[\tk_f(\ca) \;=\; \langle \;  [p]_u \;|\; [p]_u = [p_1]_u + [p_2]_u \text{ whenever } p = p_1 + p_2 \; \rangle\]
is the group generated by the unitary equivalence classes of projections in $\ca$ modulo the relations coming from addition of orthogonal projections.
Moreover, for a $*$-homomorphism  $\fdec{\phi}{\ca}{\cb}$ between  \cstar s $\ca$ and $\cb$,
\[\tk_f(\phi)([p]_u) =  [\phi(p)]_u \Mdot\]
\end{lemma} 
\begin{proof}
Let $\ca$ be a (not-necessarily-unital) \cstar.
We need to determine the kernel of $\tk_f(\pi)$ with $\pi$ the canonical projection from $\ca^+$ to $\C$. Note that $\tk_f$ in the previous sentence refers to the functor $\fdec{\tk_f}{\cat{uC^*}}{\cat{Ab}}$ defined for unital \cstar s. Therefore, we can use Lemma~\ref{lemma:kfpresentation} to perform this calculation.

All projections in $\ca^+$ are of the form $p$ or $1 - p$ for $p$ a projection in $\ca$.
From the previous lemma, the colimit group $\tilde{K}_f(\ca^+)$ is generated by elements of the form
$[p]_u$ and $[1-p]_u$ for projections $p \in \ca$.
As
$$[1-p]_u = [1]_u - [p]_u \Mcomma$$
we see that all elements of $\tilde{K}_f(\ca^+)$ can be expressed
as $\Z$-linear combinations of elements of the form $[p]_u$ with $p$ a projection in $\ca$ or $[1]_u$. Such an element is in the kernel of $\tk_f(\pi)$ if and only if the coefficient for $[1]_u$ is $0$.
Hence, $\ker\tk_f(\pi)$ is the subgroup of $\tilde{K}_f(\ca^+)$ generated by the elements $[p]_u$ for $p$ a projection in $\ca$.

The action for a $*$-homomorphism $\phi$ is clear as it is defined as a restriction of $\tk_f(\phi^+)$. 
%
%
%
\qed\end{proof}

\thmKtheory*

\begin{proof}

We are now ready to define the natural isomorphism $\fdec{\eta}{K_0}{\tk_f \circ \ck}$ as as functors $\cat{uC^*}\longrightarrow\cat{Ab}$.
For a \cstar\ $\ca$, the component $\eta_\ca$ of this natural transformation sends $[p] - [q] \in K_0(\ca)$ to $[p]_u - [q]_u \in \tk_f(\ca \otimes \ck)$, i.e. in the kernel of $\tk_f(\fdec{\pi}{(\ca \otimes \ck)^+}{\C})$.  This is well-defined, for if $[p] - [q] = [p'] - [q']$, i.e. (by Definition~\ref{def:MvNsemigroup-alternative}) if
\[[p] + [q'] + [r] = [p'] + [q] + [r] \Mcomma\]
then we can find pairwise orthogonal representatives of all these equivalence classes of projections by the remark after Definition~\ref{def:uequiv}, and show that
\[[p]_u - [q]_u = [p']_u - [q']_u \Mdot\]
Preservation of addition follows by a similar argument.

We define an inverse map to demonstrate bijectivity of $\eta_\ca$. 
A generator $[p]_u$ of $\tk_f(\ca \otimes \ck)$ is sent by $\eta_\ca^{-1}$ to $[p]$.
Since the relations from Lemma~\ref{lemma:kfpresentation-nonunital} (between $[p]_u$ and $\sum_i[p_i]_u$ whenever $p = \sum_i p_i$ and between $[p]_u$ and $[q]_u$ whenever $p$ and $q$ are unitarily equivalent) are also satisfied
by the elements $[p]$ in the $K_0$ group of $\ca$, $\eta_\ca^{-1}$ is a well-defined map.

To demonstrate the naturality of these isomorphisms, we show that the below diagram commutes.  Let $\fdec{\phi}{\ca}{\cb}$ be a unital $*$-homomorphism.

 \begin{equation*}
   \hbox{
   \xymatrix{{K_0(\ca)}  \ar@<-0.5ex>[r]^-{\eta_\ca} \ar[d]_{K_0(\phi)} & {(\tk_f \circ \ck)(\ca)} \ar[d]^{(\tk_f \circ \ck)(\phi)} \\
{K_0(\cb)} \ar@<-0.5ex>[r]_-{\eta_\cb} & {(\tk_f \circ \ck)(\cb)}
   }} \end{equation*}
  
Suppose $[p] - [q]$ is an arbitrary element of $K_0(\ca)$.
\begin{align*}
((\tk_f \circ \ck)(\phi) \circ \eta_\ca)([p] - [q]) &= ((\tk_f \circ \ck)(\phi)([p]_u - [q]_u) \\
&= \tk_f (\ck(\phi)) ([p]_u - [q]_u) \\
&=  [\ck(\phi)(p)]_u - [\ck(\phi)(q)]_u \\
&=   \eta_\cb( [\ck(\phi)(p)] - [\ck(\phi)(q)]) \\
&=   (\eta_\cb \circ K_0(\phi) ) ( [p] - [q]) 
\end{align*}
This calculation can be seen diagramatically:
 \begin{equation*}
   \hbox{
   \xymatrix{{[p] - [q]}  \ar@{|->}[r]^-{\eta_\ca} \ar@{|->}[d]_{K_0(\phi)} & {[p]_u - [q]_u} \ar@{|->}[d]^{(\tk_f \circ \ck)(\phi)} \\
{[\ck(\phi)(p)] - [\ck(\phi)(q)]} \ar@{|->}[r]_-{\eta_\cb} & {[\ck(\phi)(p)]_u - [\ck(\phi)(q)]_u}
   }} \end{equation*}
\qed\end{proof}

We have thus shown that $\fdec{K_0}{\cat{uC^*}}{\cat{Ab}}$ and $\fdec{\tk_f \circ \ck}{\cat{uC^*}}{\cat{Ab}}$ are naturally isomorphic functors.
Consequently, the extension of $\tk_f \circ \ck$ to all \cstar s via unitalisation yields a functor $\tk_f \circ \ck: \cat{C^*} \to \cat{Ab}$ that is naturally isomorphic to $\fdec{K_0}{\cat{C^*}}{\cat{Ab}}$.
Therefore, the complete operator $K_0$-functor is reconstructed solely in terms of topological $K$-theory, the finitary unitary semispectral functor, stabilisation, and unitalisation.

\section{Noncommutative topology}\label{sec:ideals}

A natural step in using extensions to directly obtain noncommutative analogues from basic topological concepts would be to establish the conjecture that extending the topological notion of closed subset leads to its algebraic generalisation: closed, two-sided ideal.  Background information and definitions for this section can be found in Appendix~\ref{app:ideals}.

We now formalise this idea.
Write $\cat{CMSLat}$ for the category of complete meet-semilattices: its objects are complete lattices and its morphisms are complete meet-semilattice homomorphisms, i.e. functions that preserve arbitrary meets.
Let $\fdec{\TT}{\cat{KHaus}}{\cat{CMSLat}}$ be the functor that assigns to a compact Hausdorff space its complete lattice of closed sets ordered by reverse inclusion (with $C_1 \leq C_2$ if and only if $C_1 \supset C_2$) and to a continuous function the complete meet-semilattice homomorphism mapping a closed set to its image under the function.\footnote{Note that a continuous function between compact Hausdorff spaces is closed, hence the direct image map preserves arbitrary meets of closed sets $\bigwedge A_i = \mathsf{cl}(\bigcup A_i)$.} Let $\tilde{\TT}$ be its extension under the unitary semispectral functor from Definition~\ref{def:unitarysemispectral-Cstar}.
Moreover, let $\fdec{\II}{\cat{uC^*}^\op}{\cat{CMSLat}}$ be the contravariant functor from the category of unital \cstar s to the category of complete meet-semilattices which sends an algebra to its complete lattice of closed, two-sided ideals and a unital $*$-homomorphism $\fdec{\phi}{\ca}{\cb}$ to the homomorphism of complete meet-semilattices   $\fdec{\II(\phi)}{\II(\cb)}{\II(\ca)}$ mapping an ideal $I \subset B$ to the ideal $\phi^{-1}(I)$ of $\ca$.
 In the commutative case, there is a correspondence between closed sets and closed ideals via \Gelfand\ duality: $\II|_{\cat{ucC^*}^\op} \simeq \TT \circ \Sigma$. This suggests the following conjecture:

\begin{conjecture}[\cite{deSilva:Ktheory}] \label{conj:idealsCstar}
  $\tilde{\TT} \simeq \II$.
\end{conjecture}

The principal theorem proved in this section is the von Neumann algebraic analogue of this conjecture:

\thmVNideals

Note that the extension here is with respect to the (von Neumann algebraic) semispectral functor of Definition~\ref{def:unitarysemispectral-vN} that assigns to a von Neumann algebra the diagram consisting of its abelian von Neumann subalgebras and restrictions of inner automorphisms.

We begin by recasting Conjecture~\ref{conj:idealsCstar} in purely algebraic terms as a correspondence between what we call \emph{total} and invariant \emph{partial} ideals of \cstar s.
We then formulate this correspondence for von Neumann algebraic ideals, which is equivalent to our principal theorem, and prove it.

\subsection{Partial and total ideals}

To prove Conjecture~\ref{conj:idealsCstar} would essentially be to demonstrate a bijective correspondence between closed, two-sided ideals of a \cstar\ $\ca$ and certain functions $\pi$ that map commutative sub-\cstar s $V$ of a \cstar\ $\ca$ to closed ideals of $V$.
To see this, note that the limit lattice $\tilde{\TT}(\ca) = (\llim \circ \lift{\TT} \circ G)(\ca)$ is the terminal cone over the diagram $(\lift{\TT} \circ G)(\ca)$.  

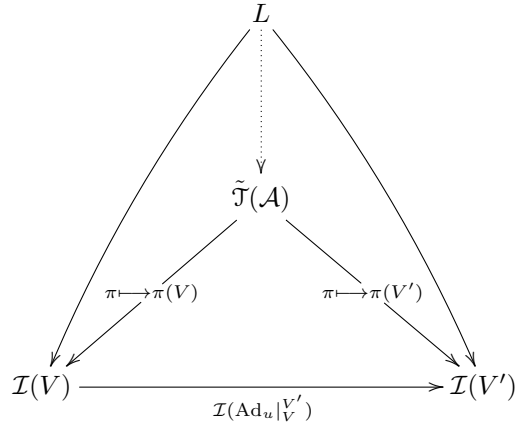
\[\xymatrix{
& & \ar@/_/[ddddll] \ar@/^/[ddddrr] L \ar@{.>}[dd] & &\\
&&&&\\
& & \ar[ddll]|{\pi \longmapsto \pi(V)} \tilde{\TT}(\ca)  \ar[ddrr]|{\pi \longmapsto \pi(V')} &&\\
&&&&\\
\II( V)  \ar[rrrr]_{\II(\ad|_{V}^{V'})}&&&& \II(V')
 }
\]

Moreover, the category $\cat{CMSLat}$
is monadic over $\cat{Set}$,
as it is the Eilenberg--Moore category of algebras of the powerset monad
 \cite[Examples 20.5(3) \& 20.10(3)]{AdamekHerrlichStrecker:JoyOfCats}.
Consequently, the forgetful functor $\fdec{U}{\cat{CMSLat}}{\cat{Set}}$ creates limits \cite[Proposition 20.12(10)]{AdamekHerrlichStrecker:JoyOfCats}.
This means that the limit of a diagram in $\cat{CMSLat}$
can be obtained by taking its limit in $\cat{Set}$---where it is given as a subset of a Cartesian product defined by equations corresponding to compatibility conditions---and equipping it with the componentwise partial order or componentwise lattice operations.

Hence, the elements of $\tilde{\TT}(\ca)$ are precisely
what we will call  the \emph{invariant partial ideals} of $\ca$: 
choices of elements $\pi(V)$ from each $\II(V)$ subject to the condition of equation \eqref{eq:conditioninvariantideals} below. 
We can thus recast Conjecture~\ref{conj:idealsCstar} (and analogously, Theorem~\ref{thm:VNideals})
in terms of a correspondence between total ideals and invariant partial ideals.

\subsubsection{Partial and total ideals of $C^*$-algebras}
All algebras and subalgebras considered throughout this section are  assumed to be unital.  By a \emph{total ideal} of a \cstar\ $\ca$, we mean a norm closed, two-sided ideal of $\ca$.

\begin{definition} \label{pi}
A \emph{partial ideal} of a  \cstar\ $\ca$ is a map $\pi$ that assigns to each commutative sub-\cstar\ $V$ of $\ca$ a closed ideal of $V$ such that $\pi(V) = \pi(V') \cap V$ whenever $V \subset V'$.
\end{definition}
Note that the last conditions can be rephrased as requiring that for any
inclusion morphism $\incfdec{\iota}{V}{V'}$, we have $\pi(V) = \II(\iota)(\pi({V'}))$, 
i.e. the following diagram commutes.
\begin{equation}\label{diag:approx}
\xymatrix{
V' & & \{*\} \ar[rr]^{* \longmapsto \pi(V')} \ar[ddrr]_{* \longmapsto \pi(V)} & & \II(V') \ar[dd]^{\II(\iota) :: I \longmapsto I \cap V}
\\
&&&&
\\
V \ar@{^{(}->}[uu]^{\iota} & & & & \II(V)
}
\end{equation}


  The concept of partial ideal was introduced by Reyes \cite{Reyes:obstructing} in the more general context of partial \cstar s.
  His definition differs slightly, but it is equivalent in our case:
  a subset $P$ of normal elements of $\ca$ such that $P \cap V$ is a closed ideal of $V$ for all commutative sub-\cstar s $V$ of $\ca$.

Partial ideals exist in abundance: every closed, left (or right) ideal $I$ of $\ca$ gives rise to a partial ideal $\pi_I$ in a natural way by choosing $\pi_{I}(V)$ to be $I \cap V$.

For example, in a matrix algebra $M_n(\C)$, the right ideal $pM_n(\C)$, for $p \in M_n(\C)$ a nontrivial projection, yields a nontrivial partial ideal of $M_n(\C)$ in this way.
  As matrix algebras are simple, it cannot be the case that these nontrivial partial ideals also arise as $\pi_I$ from a total ideal $I$. 
This raises a natural question: 

\begin{question} Which partial ideals of \cstar s arise from total ideals? \end{question}

Some partial ideals do not even arise from left or right ideals: for example, choosing arbitrary nontrivial ideals from every nontrivial commutative sub-\cstar\ of $M_2(\C)$ yields, in nearly all cases,\ nontrivial partial ideals of $M_2(\C)$.
However, a hint towards identifying those partial ideals which arise from total ideals is given by a simple observation.
If $u$ is a unitary of $\ca$, then $uIu^* = I$ for any total ideal $I \subset\ca$.
This imposes a strong condition on the partial ideal $\pi_{I}$ that arises from $I$.

\begin{definition} \label{ipi}
An {\em invariant partial ideal} $\pi$ of a \cstar\ $\ca$ is a partial ideal of $\ca$ such that, for each commutative sub-\cstar\ $V \subset \ca$ and any unitary $u \in \ca$, the conjugation by $u$ of the ideal associated to $V$ is the ideal associated to the conjugation by $u$ of $V$.  That is, 
\[ u\pi(V)u^* = \pi(uVu^*) \]
\end{definition}
Equivalently, if we write $\fdec{\ad}{\ca}{\ca}$ for the inner automorphism given by conjugation by $u$, i.e. $a \longmapsto u a u^*$, the condition above reads 
\[\ad(\pi(V)) = \pi({\ad(V)}) \Mdot\]

Imposing this invariance condition on partial ideals is equivalent to extending the requirement on $\pi$ of Diagram \eqref{diag:approx} from inclusions $\fdec{\iota}{V}{V'}$
to all $*$-homomorphisms $\fdec{\ad|_V^{V'}}{V}{V'}$ arising as a restriction of the domain and codomain of an  inner automorphism.
An invariant partial ideal is precisely a choice of $\pi(V) \in \II(V)$ for each commutative sub-\cstar\ $V$ of $\ca$ such that whenever there is a morphism $\fdec{\ad|_V^{V'}}{V}{V'}$ as above, then
\begin{equation}\label{eq:conditioninvariantideals}
\pi(V) = \II(\ad|_{V}^{V'})(\pi({V'})) = \mathrm{Ad}_{u^*}(\pi(V')) \cap V = u^*\pi(V')u \cap V \Msemicolon
\end{equation}
i.e. the following diagram commutes.
\[
\xymatrix{
V' & & \{*\} \ar[rr]^{* \longmapsto \pi(V')} \ar[ddrr]_{* \longmapsto \pi(V)} & & \II(V') \ar[dd]^{\II(\ad|_{V}^{V'})}
\\
&&&&
\\
V \ar[uu]_{\ad|_{V}^{V'}} & & & & \II(V)
}
\]

%
%

Thus, we arrive at the following conjecture:

\begin{conjecture}[Reformulation of Conjecture~\ref{conj:idealsCstar}] \label{partialidealsc}
A partial ideal of a \cstar\ arises from a total ideal if and only if it is an invariant partial ideal.
Consequently, the map $I \longmapsto \pi_I$ is a bijective correspondence between total ideals and invariant partial ideals.
\end{conjecture}
Note that the first part of the statement says that the map  $I \longmapsto \pi_I$ is surjective onto the invariant partial ideals.
The second part of the statement follows easily from this, since injectivity of this map is obvious: the left inverse is given by mapping an invariant partial ideal of the form $\pi_I$ to the linear span of $\bigcup_{V}\pi(V)$, which is equal to $I$ itself.

\subsubsection{Partial and total ideals of von Neumann algebras}

A total ideal of a von Neumann algebra is, as in Definition~\ref{idealdef}, an ultraweakly closed, two-sided ideal.
One may define \emph{partial ideal} (resp. \emph{invariant partial ideal}) for a  von Neumann algebra by replacing in Definition~\ref{pi} (resp. Definition~\ref{ipi})  the occurrences of ``commutative sub-\cstar'' with ``commutative sub-von Neumann algebra'' and ``closed ideal'' with ``ultraweakly closed ideal''.
As before, a total ideal $I$ determines an invariant partial ideal $\pi_I$ in the same way, and the map $I \longmapsto \pi_I$ is injective.


Besides its intrinsic interest, establishing the analogue of Conjecture~\ref{partialidealsc} for von Neumann algebras
provides some measure of evidence for the original conjecture's verity,
and it may be the case that its proof can be adapted to show that the original conjecture holds for a large class of---or perhaps all---\cstar s.

\begin{theorem}[Principal theorem of section] \label{main}
A partial ideal of a von Neumann algebra arises from a total ideal if and only if it is an invariant partial ideal.
Consequently, the map $I \longmapsto \pi_I$ is a bijective correspondence between total ideals and invariant partial ideals.
\end{theorem}

Total ideals of a von Neumann algebra $\ca$ are in bijective correspondence with central projections $z$ of $\ca$: every total ideal $I$ is of the form $z\ca$ for a unique $z$ \cite{AlfsenShultz1} (Theorem~\ref{idealcentral}).  This allows us to rephrase the theorem in terms of projections, which are vastly more convenient to work with.

\begin{definition}
A \emph{consistent family of projections} of a von Neumann algebra $\ca$ is a map $\Pf$ that assigns to each commutative sub-von Neumann algebra $V$ of $\ca$ a projection in $V$ such that:
\begin{enumerate}
  \item\label{item:dasein}
  for any $V$ and $V'$ such that $V \subset V'$, $\P{V}$ is the largest projection in $V$ which is less than or equal to $\P{V'}$, i.e.
  \[\P{V} = \sup\setdef{q \text{ is a projection in $V$}}{ q \leq \P{V'}} \Mdot\]
\end{enumerate}
An \emph{invariant family of projections} is such a map further satisfying
\begin{enumerate}\setcounter{enumi}{1}
\item
  for any unitary element $u \in \ca$, $\P{uVu^*}$ = $u\P{V}u^*$.
\end{enumerate}
\end{definition}

The correspondence between total ideals and central projections yields correspondences between partial ideals (resp. invariant partial ideals) and consistent (resp. invariant) families of projections.  Therefore, we shall establish Theorem~\ref{main} by proving the equivalent statement below.
Just as was the case for ideals, any projection $p$ determines a consistent family of projections $\Ppf$ defined by choosing $\Pp{V}$ to be the largest projection $p$ in $V$ which is less than or equal to $p$.
For a central projection $z$, $\Pzf$ turns out to be an invariant family.
In the opposite direction, any consistent family of projections $\Pf$ gives a central projection  $\P{\cz(\ca})$ where $\cz(\ca)$ is the centre of $\ca$.

\begin{theorem}[Principal theorem of section, reformulated] \label{main-reformulated}
A consistent family of projections of a von Neumann algebra arises from a central projection if and only if it is an invariant family of projections.
Consequently, the maps $z \longmapsto \Pzf$ and $\Pf \longmapsto \P{\cz(\ca)}$ define a bijective correspondence between central projections and invariant families of projections.
\end{theorem}

\subsection{Technical preliminaries}

\subsubsection{Little lemmata}
In proving our main result, we shall make use of some simple properties of consistent families of projections which we record here as lemmata for clarity.

\begin{lemma}\label{lemma:approxprops}
  Let $\ca$ be a von Neumann algebra and $\Pf$ be a consistent family of projections in $\ca$.  Suppose $V$ and $V'$ are commutative sub-von Neumann algebras of $\ca$ with $V \subset V'$. Then:
  \begin{enumerate}[(i)]
    \item\label{item:monotone} $\P{V} \leq \P{V'}$;
    \item\label{item:approxbest} if $p \in V$ and $p \leq \P{V'}$, then $p \leq \P{V}$;
    \item\label{item:approxbesteq} in particular, if $\P{V'} \in V$, then $\P{V'} = \P{V}$.
  \end{enumerate}
\end{lemma}
\begin{proof}
  Properties \eqref{item:monotone} and \eqref{item:approxbest} are simple consequences of the requirement in the definition of consistent family of projections that $\P{V}$ is the largest projection in $V$ smaller than $\P{V'}$. Property \eqref{item:approxbesteq} is a particular case of \eqref{item:approxbest}.
\qed\end{proof}

Given a commutative subset $X$ of a von Neumann algebra $\ca$,
denote by $\V{X}$ the commutative sub-von Neumann algebra of $\ca$ generated by $X$ and the centre $\cz(\ca)$,
i.e. $\V{X} = (X \cup \cz(\ca))''$. Note that $\V{\emptyset}=\cz(\ca)$.
Given a nonempty finite commutative set of projections $\{p_1,\ldots,p_n\}$, we write $\V{p_1, \ldots, p_n}$ for $\V{\{p_1, \ldots, p_n\}}$.

\begin{lemma}\label{lemma:approxsupVm}
 Let $\ca$ be a von Neumann algebra and $\Pf$ a consistent family of projections in $\ca$.
 Let $M$ be a commutative set of projections in $\ca$ and write $s$ for the supremum of the projections in $M$.
 If $\P{\V{m}} \geq m$ for all $m \in M$, then $\P{\V{s}} \geq s$.
\end{lemma}
\begin{proof}
    For all $m \in M$, since $\V{m} \subseteq \V{M}$, we have
  \[\P{\V{M}} \geq \P{\V{m}} \geq m\]
  by Lemma~\ref{lemma:approxprops}-\eqref{item:monotone} and the assumption that $\P{\V{m}} \geq m$.
  Hence, $\P{\V{M}}$ is at least the supremum of the projections in $M$, i.e. $\P{\V{M}} \geq s$.
  Now, note that $s \in \V{M}$ as it is a supremum of projections in $\V{M}$, hence $\V{s} \subset \V{M}$.
  From this and $s \in \V{s}$, we conclude by Lemma~\ref{lemma:approxprops}-\eqref{item:approxbest} that $s \leq \P{\V{s}}$.
\qed\end{proof}
 

\subsubsection{Partial orthogonality}
We introduce the following notion, which will prove useful in establishing our main result.
Note that, given a projection $p$, we write $p^\perp$ for the projection $1 - p$.

\begin{definition}
Two projections $p$ and $q$ in a von Neuman algebra are {\em partially orthogonal} whenever there exists a central projection $z$ such that $zp$ and $zq$ are equal while $z^\perp p$ and $z^\perp q$ are orthogonal.
\end{definition}

A set of projections is said to be {\em partially orthogonal} whenever any pair of projections in the set is partially orthogonal. This can include pairs of repeated elements, as any projection is trivially partially orthogonal to itself.

Note that partially orthogonal projections necessarily commute.
Moreover, if $p_1$ and $p_2$ are partially orthogonal, then so is the pair $zp_1$ and $zp_2$ for any central projection $z$.
We will require in the sequel the following simple lemma:

%

\begin{lemma}\label{polemma}
In a von Neumann algebra, let $p_1$ and $p_2$ be projections and $z$ be a central projection such that
$zp_1$ and $zp_2$ are partially orthogonal and $z^\perp p_1$ and $z^\perp p_2$ are partially orthogonal. 
Then $p_1$ and $p_2$ are partially orthogonal.
\end{lemma}
\begin{proof}
As $zp_1$ and $zp_2$ are partially orthogonal, there exists a central projection $y$ such that
\[ yzp_1 = yzp_2 \;\; \text{ and } \;\;         y^\perp zp_1 \perp y^\perp zp_2 \Mdot\]
Similarly, as $z^\perp p_1$ and $z^\perp p_2$ are partially orthogonal, there exists a central projection $x$ such that
\[ xz^\perp p_1 = xz^\perp p_2 \;\; \text{ and } \;\; x^\perp z^\perp p_1 \perp x^\perp z^\perp p_2 \Mdot\]
Summing both statements above, we conclude that
\[(yz+xz^\perp )p_1 = (yz+xz^\perp )p_2 \;\; \text{ and } \;\; (y^\perp z+x^\perp z^\perp ) p_1 \perp (y^\perp z + x^\perp z^\perp )p_2 \Mcomma\]
where $yz+xz^\perp $ is a central projection and $(yz+xz^\perp )^\perp  = y^\perp z+x^\perp z^\perp $.
So, $p_1$ and $p_2$ are partially orthogonal.
\qed\end{proof}

\subsubsection{Main lemma}
When comparing projections, we write $\leq$ to denote the usual order on projections, $\leqMvN$ for the order up to Murray--von Neumann equivalence, and $\lequ$ for the order up to unitary equivalence.

The following lemma is one of the main steps of the proof.
The idea is to start with a projection $q$ in a von Neumann algebra and to
cover, as much as possible, its central carrier $C(q)$ by a commutative subset of the unitary orbit of $q$.
The lemma states that, in order to cover $C(q)$ with projections from the unitary orbit of $q$,
it suffices to take a commutative subset, $M$, and (at most) one other projection, $uqu^*$,
which is above the remainder $C(q) - \sup M$.
In other words, the remainder from what can be covered by a commutative set
 $M$ is smaller than or equal to $q$ up to unitary equivalence.

\begin{lemma}\label{lemma:main}
  Let $q$ be a projection in a von Neumann algebra $\ca$.
  Then there exists a set $M$ of projections in $\ca$ such that:
\begin{enumerate}[(i)]
\item\label{mainlemmaitem:first}
$q \in M$;
\item
$M$ is a subset of the unitary orbit of $q$;
\item\label{mainlemmaitem:penultimate}
$M$ is a commutative set;
\item
the supremum $s$ of $M$ satisfies $\srem \lequ q$ where $\srem = C(s)-s = C(q) - s$.
\end{enumerate}
\end{lemma}
\begin{proof}
Let $O$ be the unitary orbit of $q$.  The partially orthogonal subsets of $O$  which contain $q$ form a poset under inclusion.
Given a chain in this poset, its union is partially orthogonal: any two projections in the union must appear together somewhere in one subset in the chain and are thus partially orthogonal. Hence, by Zorn's lemma, we can construct a maximal partially orthogonal subset $M$ of the unitary orbit of $q$ such that $q \in M$. Clearly, $M$ satisfies conditions \eqref{mainlemmaitem:first}--\eqref{mainlemmaitem:penultimate}. 

Denote by $s$ the supremum of the projections in $M$.
Its central carrier $C(s)$ is equal to the central carrier $C(q)$ of $q$.  This is because $C(-)$ is constant on unitary orbits and $C(\sup_{m \in M}m) = \sup_{m \in M}C(m)$.
We now need to show that $\srem \lequ q$.

By the comparison theorem for projections in a von Neumann algebra (Theorem~\ref{comparison}), there is a central projection $y$ such that
\begin{equation}\label{eq:applycompthm_y}
y\srem \geqMvN  yq \;\;\text{ and }\;\; y^\perp \srem \leqMvN y^\perp q \Mdot
\end{equation}
We show that this central projection can be taken to be below $C(q)$.
Consider the central projection $z = y\, C(q)$. Then,
for any projection $r \leq C(q)$,
i.e. $C(q) r = r$,
we have
\[z r  = y C(q) r = y r
\Mand
z^\perp r = r - zr = r - yr = y^\perp r \Mdot\]
Since both $\srem, q \leq C(q)$, one can rewrite \eqref{eq:applycompthm_y} as
\[z\srem \geqMvN  zq \;\;\text{ and }\;\; z^\perp \srem \leqMvN z^\perp q \Mcomma\]
where $z = y\, C(q) \leq C(q)$.


By Proposition~\ref{prop:mvNimpliesunitary}, as $q$ and $\srem$ are orthogonal,
there are unitaries that witness these order relationships.
That is, there are unitaries $u$ and $v$ such that 
\begin{equation}\label{eq:six}
z\srem\geq z(uqu^*)  \;\;\text{ and }\;\; z^\perp \srem \leq z^\perp (vqv^*) \Mdot
\end{equation}
We will show that $z$ vanishes and thus conclude that $\srem \leq vqv^*$.

Define $u_z$ to be the unitary $zu + z^\perp 1$ which acts as $u$ within the range of $z$ and as the identity on the range of $z^\perp $.
We first establish that $u_zqu_z^*$ and $m$ are partially orthogonal for every $m \in M$.  

Let $m \in M$. As $M$ was defined to be a partially orthogonal set of projections and $q \in M$,
we know that $q$ and $m$ are partially orthogonal, and thus that $z^\perp q$ and $z^\perp m$ are partially orthogonal.
However, as $z^\perp u_z = z^\perp $, we may express this as: $z^\perp (u_zqu_z^*)$ and $z^\perp m$ are partially orthogonal.
Additionally, on the range of $z$, we have that
\[z(u_z q u_z^*) = z(uqu^*) \leq z\srem \;\;\text{ and }\;\; zm \leq zs \Mcomma\]
implying that $z(u_z q u_z^*)$ and $zm$ are orthogonal, hence partially orthogonal.
Putting both parts together, we have that $z^\perp u_z q u_z^*$ and $z^\perp m$ are partially orthogonal and that $z(u_z q u_z^*)$ and $zm$ are partially orthogonal.
We may thus apply Lemma~\ref{polemma} and conclude that $u_z q u_z^*$ and $m$ are partially orthogonal as desired.

Having established that $u_z q u_z^*$ is partially orthogonal to all the projections in $M$, it follows by maximality of $M$ that $u_z q u_z^* \in M$.
Hence, \[z u_z q u_z^* \leq u_z q u_z^* \leq \sup M = s \Mdot\]
Yet, by construction, 
\[z u_z q u_z^* = zuqu^* \leq z\srem \leq \srem\Mcomma\] and so $z u_z q u_z^*$ must be orthogonal to $s$.
Being both contained within and orthogonal to $s$, $z u_z q u_z^*$ must vanish. Therefore, the unitarily equivalent projection $zq$ must also vanish.
Now, $zq = 0$ implies that $z^\perp$ covers $q$. But since $z^\perp$ is a central projection, it must also cover the central carrier of $q$, i.e.  $C(q) \leq z^\perp$. We thus have $z \leq C(q) \leq z^\perp$, forcing $z$ to be zero.

We may finally conclude, by (\ref{eq:six}), that $\srem \leq vqv^*$.
\qed\end{proof}

\subsection{Main theorem}

Theorem~\ref{main-reformulated}, and thus our principal result, Theorem~\ref{main},
will follow as an immediate corollary of the following theorem. 
\begin{theorem}
  In a von Neumann algebra $\ca$, any invariant family of projections $\Pf$ arises from a central projection, i.e. $\Pf$ is equal to $\Pzf$ for the  central projection $z = \Pf(\cz(\ca))$.
\end{theorem}
\begin{proof}
Let $\Pf$ be an invariant family of projections.
Suppose $W $ is a commutative sub-von Neumann algebra of $\ca$ which contains the centre $\cz(\ca)$, and let $q$ be the projection $\P{W}$.
We claim that $q$ is, in fact, equal to its own central carrier $C(q)$, and thus central.
As $q \leq C(q)$ is true by definition, we must show that $q \geq C(q)$.

We start by applying Lemma~\ref{lemma:main} to  $q$. Let $M$ denote the resulting commuting set of projections in the unitary orbit of $q$,
$s$ denote the supremum of the projections in $M$, and $\srem$ denote $C(q) - s = C(s) - s$. From the lemma, we know that $\srem \lequ q$, i.e. there exists a unitary $u$ such that $\srem \leq u q u^*$.


First note that, since $\V{q} \subset W$ and $q \in \V{q}$,  by Lemma~\ref{lemma:approxprops}-\eqref{item:approxbesteq}, we have that  $\P{\V{q}} = q$.
Then, by unitary invariance of the family of projections, for every $m \in M$ we have that $\P{\V{m}} = m$.
Hence, we can apply Lemma~\ref{lemma:approxsupVm} to conclude that $\P{\V{s}}\geq s$.
 We also conclude, again by unitary invariance of $\Pf$, that $\P{\V{uqu^*}} = uqu^* \geq \srem$.

Now, note that $uqu^*$ and $\srem$ commute.
Moreover, $\V{s} = \V{\srem}$ since $\srem = C(q) - s$, hence $s_R$ is in the algebra generated by $s$ and the centre, and vice-versa.
So, there is a commutative sub-von Neumann algebra
$\V{s, uqu^*} \supseteq \V{s}, \V{uqu^*}$.
By Lemma~\ref{lemma:approxprops}-\eqref{item:monotone} and the two conclusions of the preceding  paragraph, we then have
\[\P{\V{s, uqu^*}} \geq \P{\V{s}} \vee \P{\V{uqu^*}} \geq s \vee \srem = C(q) \Mdot\]
But, since $C(q) \in \V{uqu^*}$ by virtue of it being contained in the centre, we can apply Lemma~\ref{lemma:approxprops}-\eqref{item:approxbest} to find that $\P{\V{uqu^*}} \geq C(q)$.
Finally, by unitary invariance, \[q = \P{\V{q}} \geq u^*C(q)u = C(q) \Mcomma\] concluding the proof that $q$ is central.


We have shown that the projection $\P{W}$ is central for every commutative sub-von Neumann algebra $W$ containing the centre $\cz(\ca)$.
By Lemma~\ref{lemma:approxprops}-\eqref{item:approxbesteq}, this means that $\P{W}$ is equal to $\P{\cz(\ca)}$, the projection chosen at the centre, for all such $W$. In turn, this determines the image of $\Pf$ on all commutative sub-von Neumann algebras $W'$ as 
\[\P{W'} = \sup\setdef{p \text{ is a projection in $W'$}}{p \leq \P{V_{W'}} = \P{\cz(\ca)}} \Mcomma\]
and we find that $\Pf$ must be equal to $\Pf_{\P{\cz(\ca)}}$.
\qed\end{proof}

\section{Conclusions}\label{sec:conclusions}

In this work, we have argued that a nonstandard---but nevertheless foundationally important---notion of quantum state space is
the dual object of a noncommutative algebra with respect to state-observable duality,
and noted that such spaces are the objects of study within noncommutative geometry of \cstar s.
In noncommutative operator geometry \cite{connes}, \Gelfand--\Naimark\ duality justifies interpreting noncommutative \cstar s as representing the algebra of observables on a hypothetical geometric space; it further provides a heuristic method of translating topological concepts into algebraic language.
We further argued that an explicitly geometric construction of this notion of quantum state space
should provide simple means for the direct extension of topological concepts to noncommutative generalisations in a manner coinciding with the constructions of noncommutative geometry.

Our ansatz for a geometric space associated to a noncommutative algebra comes from the \emph{spectral presheaf} construction of Hamilton--Isham--Butterfield \cite{oldtopos3}, which accounts for the essential nonclassicality of quantum theory expressed by the Bell--Kochen--Specker theorem \cite{bell1966,KochenSpecker} by associating a classical state space to each context (commutative subalgebra) of the algebra of observables.  We offer an alternate interpretation of the collection of classical state spaces as the collection of tractable quotient spaces of the noncommutative space represented by the algebra of observables.  We show how functorial associations of spatial diagrams to algebras yields automatic methods of extending functors defined on topological spaces to ones defined on \cstar s.  After modifying the spectral presheaf by including data related to inner automorphisms, we consider the extensions of two functors: the topological $K$-functor \cite{atiyah1967k} and the functor $\TT$ that assigns to a space its lattice of closed subsets.

In the former case, we give a novel definition of operator $K$-theory, $K_0$, in terms of a colimit of vector bundles over the finite quotient spaces of stable noncommutative spaces. This formally aligns very closely with the extension $\tilde{K}$ of the topological $K$-functor.
Specifically, we have shown that operator $K_0$ of a \cstar\ $\ca$ corresponds to the extension of the topological $K$-functor from the spaces corresponding to finite-dimensional subalgebras of the stabilisation of $\ca$.
While $K_0 \simeq \tilde{K}$ holds for finite-dimensional \cstar s, whether it holds in general
(or whether $K_0 \simeq \tilde{K} \circ \ck$) remains an open question.

In the latter case,
we establish a bijective correspondence between ideals of a von Neumann algebra
and what could be thought of as clopen subsets of its associated spatial diagram.
More formally, we display a natural isomorphism $\II_\mathsf{W} \simeq \tilde{\TT_\mathsf{W}}$ between
the functor $\II_\mathsf{W}$ that assigns to a von Neumann algebra its lattice of ultraweakly closed, two-sided ideals and the extension of the functor $\TT_\mathsf{W}$ mapping a (hyperstonean) topological space to its lattice of clopen sets.
This theorem is the von Neumann algebraic analogue of the conjecture that $\II \simeq \tilde{\TT}$ where $\II$ is the ideal lattice functor and $\TT$ the functor mapping a topological space to its lattice of closed sets.

As a consequence of the von Neumann algebraic theorem, the $C^*$-algebraic conjecture holds for all finite-dimensional \cstar s.  The question of whether it holds for all \cstar s remains open.  An immediate question is whether the conjecture holds for $AF$-algebras, i.e. those that arise as limits of finite-dimensional \cstar s \cite{bratteli1972inductive}. This would follow immediately from a proof of the continuity of $\tilde{\TT}$.  Another tack would be to prove the whole conjecture directly by using the proof of the von Neumann algebraic version as a guide.  Indeed, one might still be able to reduce the question to one about projections by working in the enveloping von Neumann algebra $\ca^{**}$ of a \cstar\ $\ca$. In this setting, the total ideals of a \cstar\ $\ca$ correspond to certain total ideals of the enveloping algebra $\ca^{**}$ \cite{AlfsenShultz1}: those that correspond to open central projections.  One would have to find a correspondence between open central projections of $\ca^{**}$ and certain families of open projections obeying a restricted form of unitary invariance.
This might also be formulated directly at the $C^*$-algebraic level by considering approximate identities for ideals as playing the role of central projections.

Proving that $\ci \simeq \tilde{\TT}$ would establish a strong relationship between the topologies of the geometric object $G(\ca)$ and $\Prim(\ca)$, the primitive ideal space of $\ca$: we would be able to recover the lattice of the hull-kernel topology on $\Prim(\ca)$ as the limit of the topological lattices of the object $G(\ca)$.  Establishing this conjecture would allow considering $G$ to be an enrichment of $\Prim$.  $\Prim$ is a  $C^*$-algebraic variant of the ring-theoretic spectrum  whose hull-kernel topology provides the basis for sheaf-theoretic methods in ring theory. 

Our proposal for a notion of noncommutative spectrum bears structural similarity with related functional and order-theoretic constructions that represent noncommutative algebras as a commutative fragment augmented with a unitary group action \cite{mayet1992orthosymmetric,heunen2014active,hamhalter2014automorphisms,Doering2012:FlowsOnGeneralisedGelfandSpectra,Barbosa:DPhil}.
Novel abstract frameworks for understanding noncommutative algebras in terms of commutative subalgebras have appeared since our starting to work on this line of research, notably \cite{heunen2014active,FloriFritz:Almost}, which include some results with a similar flavour to ours.
Understanding the relationships between these approaches, and the question of whether some synthesis of them might better clarify noncommutative geometry is an important line of future work.

As discussed in more detail in Section~\ref{ssec:motivation}, to establish a concrete duality, it would be important to characterise which diagrams of spaces arise as spatial diagrams of a noncommutative algebra, and which maps between them correspond to unital $*$-homomorphisms.  To facilitate computations, some notion of a sub-spatial diagram `cover', analogous to a tractable choice of charts for a manifold, may be needed. 
Another key step would be to recover a noncommutative algebra from its spatial diagram together with some extra data.

There are some topological concepts that are usually understood to be inextendable to the commutative setting. The simplest example is provided by the notion of points. We have seen that the forgetful functor to the category $\text{Set}$, which associates to a space its set of points, has a trivial extension to the noncommutative setting, and that this corresponds to the Bell--Kochen--Specker theorem from quantum foundations. So, the topological notion of points does not survive our process of translation, in agreement with the common intuition that noncommutative spaces have \textit{no points}.
On the other hand, there are interesting concepts in the noncommutative setting that become trivial when restricted to the commutative case, e.g. Tomita--Takesaki theory. Clearly, such intrinsically noncommutative concepts cannot be obtained by the process of extension outlined in this article, via diagrams of topological spaces corresponding to commutative subalgebras. However, they can provide valuable guidance in determining what extra data must be adjoined  to the diagram of topological spaces in order to recover a noncommutative algebra, since that data must be used in an essential way to define those concepts.

It would also be interesting to calculate the spatial diagram explicitly for some special examples.  One promising possibility is the canonical commutation relations algebra, which is closely connected with quantisation and thus physically very significant.  In this case, the Krichever--Mulase classification \cite{mulase1990category} of  certain commutative subalgebras of $\C[[x]][\partial]$   provides a potentially highly useful roadmap.  Another possibly tractable class of algebras for computations are those which arise as crossed product algebras, wherein a group action on a \cstar\ is embedded in a larger \cstar\ such that the action is realised as a group of inner automorphisms.  This class includes within it the important example of noncommutative tori, the computation of whose $K$-theory was considered a very difficult problem \cite{rieffel1981c}.

The idea of looking at commutative quotients is a very general one and could perhaps be applied to analyse other sorts of noncommutative algebras other than \cstar s.  The ideas outlined above might be applied to any duality involving a category of geometric objects and a category of commutative algebras.

\begin{acknowledgements}
It is our pleasure to thank Samson Abramsky, Bob Coecke, Andreas D\"oring, George Elliott, Chris Heunen, Kobi Kremnitzer, Klaas Landsman, Brent Pym, Jamie Vicary, Manuel Reyes, and various participants of the Operator Algebras Seminar at the Fields Institute for their guidance, encouragement, and mathematical insights.
We also thank the anonymous referees for valuable suggestions and insights.

This work was carried out in part while ND was based at the Department of Computer Science, University of Oxford,
and finalised while both authors visited the Simons Institute for the Theory of Computing (supported by the Simons Foundation) at the University of California, Berkeley, as participants of the Logical Structures in Computation programme.

Financial support from the following is gratefully acknowledged:
the Clarendon Fund (ND);
Merton College, Oxford (ND);
the Oxford University Computing Laboratory (ND);
the Natural Sciences and Engineering Research Council of Canada (ND);
Marie Curie Initial Training Network `MALOA -- From MAthematical LOgic to Applications', PITN-GA-2009-238381 (RSB);
FCT -- Funda\c{c}\~{a}o para a Ci\^encia e Tecnologia (Portuguese Foundation for Science and Technology), PhD grant SFRH/BD/94945/2013 (RSB);
John Templeton Foundation, Grant ID-35740, `Categorical Unification' (RSB);
Oxford Martin School James Martin Program on Bio-inspired Quantum Technologies (RSB);
Engineering and Physical Sciences Research Council, EP/N018745/1 \& EP/N017935/1, `Contextuality as a Resource in Quantum Computation' (ND \& RSB);
and the Simons Institute for the Theory of Computing (ND \& RSB).
\end{acknowledgements}

\appendix

\section{Concrete colimit construction}\label{app:colimits}

Here, we give a concrete construction of the generalised colimit functor $\fdec{\colim}{\codiag(\catC)}{\catC}$
in terms of coequalisers of coproducts.

Recall that the colimit of a functor $D$ from a category $\catA$ to a cocomplete category $\catC$ can be expressed as a coequaliser of two coproducts \cite[p. 355]{MacLaneMoerdijk}:
\[
\xymatrix{\coprod_{\fdec{u}{i}{j}}{D(\dom u)} \ar@<0.5ex>[r]^-\theta \ar@<-0.5ex>[r]_-\tau & \coprod_{i}{D(i)}}
\]
The first coproduct is over all morphisms $\fdec{u}{i}{j}$ of $\catA$ and the second is over all objects $i$ of $\catA$.
We denote the canonical injections for these coproducts by
\[
\fdec{\lambda_v}{D(\dom v)}{\coprod_{\fdec{u}{i}{j}}{D(\dom u)}}
\]
and
$$\fdec{\kappa_j}{D(j)}{\coprod_{i}{D(i)}} \Mdot$$
The morphisms $\theta$ and $\tau$ can be defined by specifying their compositions with the $\lambda_v$:
\[
\theta \circ \lambda_v = \kappa_{\dom v}
\Mand
\tau \circ \lambda_v = \kappa_{\cod v} \circ D(v)
\]
The advantage of this coequaliser presentation of the colimit is that we may determine a $\catC$-morphism between the colimits of two functors $\fdec{D}{\catA}{\catC}$ and $\fdec{E}{\catB}{\catC}$ by specifying a natural transformation between their coequaliser diagrams. That is, by giving its components,  $\catC$-morphisms $N$ and $M$ such that the following diagrams commute:
\begin{equation*}
   \xymatrix{\coprod_{\fdec{u}{i}{j}}{D(\dom u)} \ar@<0.5ex>[r]^-\theta  \ar[d]^N & \coprod_{i}{D(i)} \ar[d]^M \\
\coprod_{\fdec{u'}{i'}{j'}}{E(\dom u')} \ar@<0.5ex>[r]^-{\theta'}  & \coprod_{i'}{E(i')}
\\
\coprod_{\fdec{u}{i}{j}}{D(\dom u)}  \ar@<-0.5ex>[r]_-\tau \ar[d]^N & \coprod_{i}{D(i)} \ar[d]^M \\
\coprod_{\fdec{u'}{i'}{j'}}{E(\dom u')} \ar@<-0.5ex>[r]_-{\tau'} & \coprod_{i'}{E(i')}}
   \end{equation*}
where $i'$ and $u'$ range over all objects and all morphisms of $\catB$, respectively, while $i$ and $u$ are as above.
We denote the canonical injections into the coproducts for $E$ by $\lambda_{v'} '$ and $\kappa_{j'}'$.

Given a $\codiag(\catC)$-morphism $(f, \eta)$ between $D$ and $E$ we define $N$ and $M$ by giving their compositions with the canonical injections:
\[
N \circ \lambda_v = \lambda_{f(v)}' \circ \eta_{\dom v} \Mand
M \circ \kappa_j = \kappa_{f(j)}' \circ \eta_j
\]

It is straightforward to verify that the above diagrams commute, that is, that $\theta' \circ N = M \circ \theta$ and that $\tau' \circ N = M \circ \tau$, by computing the composition of these maps with the $\lambda_v$.
The $\catC$-morphism assigned by $\colim$ to $(f, \eta)$ is then defined to be the morphism that is induced by the natural transformation (whose components are $N$ and $M$) between the coequaliser diagrams for the colimits of $D$ and $E$.  

Functoriality of $\colim$ is then straightforwardly verified by computing the compositions of the components of the natural transformations induced by $(f, \eta)$ and $(g, \mu)$ and seeing that the resulting natural transformation is the same as the one induced by $(g \circ f, \mu_f \circ \eta)$.

\section{Topological and $C^*$-algebraic $K$-theory}\label{app:ktheory}

\subsection{Topological $K$-theory}

Topological $K$-theory, invented by Atiyah--Hirzebruch \cite{atiyah1967k} after Grothendieck \cite{grothendieck1968classes},  is an extraordinary cohomology theory,
i.e. satisfies the Eilenberg--Steenrod axioms \cite{eilenberg1945axiomatic} except the dimension axiom.
It is determined by a sequence of contravariant functors from $\cat{KHaus}$ to $\cat{Ab}$.
%
%
%
After  early successes, including the solution to the classical problem of determining how many linearly independent vector fields can be constructed on $\mathbf{S}^n$ \cite{adams1962vector}, the subject blossomed to include algebraic and analytic versions.  The core idea is to describe the geometry of a space by algebraic information about the possible vector bundles over it.  Here, we briefly review its definition.    Its generalisation to \cstar s, operator $K_0$, is a key tool of noncommutative geometry and will be outlined in the next subsection.

\begin{definition}
For a compact Hausdorff space $X$, the \emph{vector bundle monoid} $V(X)$ is the set of isomorphism classes of complex vector bundles over $X$ with the abelian addition operation of fibrewise direct sum:  $[E] + [F] = [E \oplus F]$. A continuous function $\fdec{f}{X}{Y}$  yields a monoid homomorphism $\fdec{V(f)}{V(Y)}{V(X)}$ by the pullback of bundles.
That is, if $\fdec{p}{E}{Y}$ is a bundle over $Y$, the bundle $f^*E$ is a bundle over $X$ given by the projection to $X$ of
$$\setdef{(x,v) \in X \times E}{f(x) = p(v)}$$
and  $V(f)([E]) = [f^*E]$.
 This defines a functor $\fdec{V}{\cat{KHaus}^\op}{\cat{AbMon}}$.
\end{definition}

\begin{definition}For an abelian monoid $M$, the \emph{Grothendieck group} of $M$,  $\cg(M)$, is the abelian group $(M \times M)\slash{\sim}$ where $\sim$ the equivalence relation given by
\[(a, b) \sim (c, d) \Miff \Exists{e \in M} a + d + e =b +c +e \Mdot\]
For an abelian monoid homomorphism $\fdec{\phi}{M}{N}$, the abelian group homomorphism $\fdec{\cg(\phi)}{\cg(M)}{\cg(N)}$ is given by by $\cg(\phi)([(a, b)])  = [(\phi(a),\phi(b))]$.  This defines a functor $\fdec{\cg}{\cat{AbMon}}{\cat{Ab}}$.
\end{definition}

Intuitively, an element $[(a,b)]$ of $\cg(M)$ can be thought of as a formal difference $a-b$ of elements of $M$.
With this interpretation in mind, it is easy to see that $\cg(M)$ is indeed a group, with addition given componentwise, neutral element $[(0,0)]$, and the inverse of $[(a,b)]$ equal to $[(b,a)]$. Moreover, there is a monoid homomorphism $\fdec{i}{M}{\cg(M)}$ given by $a \longmapsto [(a,0)]$.
As an example, the Grothendieck group of the additive monoid of natural numbers (including zero) is the additive group of the integers.

The Groethendieck group functor $\cg$ is an explicit presentation of the group completion functor,
the left adjoint to the forgetful functor from $\cat{Ab}$ to $\cat{AbMon}$.
This means that  $\cg(M)$ is the `most general' group containing a homomorphic image of $M$, in the sense that it satisfies the universal property that 
any monoid homomorphism from $M$ to an abelian group factors uniquely through the homomorphism $\fdec{i}{M}{\cg(M)}$.


\begin{definition}
The \emph{topological $K$-functor} $\fdec{K}{\cat{KHaus}^\op}{\cat{Ab}}$ is $\cg \circ V$.
\end{definition}

From the topological $K$-functor, one can easily construct the full sequence of functors $K_n$ for $n \in \NN$.

\begin{definition}
The \emph{suspension functor} $S$ maps the category of compact Hausdorff space to itself by sending a space $X$ to the quotient space
$$X \times [0,1] / \{(x,0) \sim (x',0)\text{ and }(x,1) \sim (x',1)\text{ for all } x,x' \in X\}$$
and a continuous function $\fdec{f}{X}{Y}$ to the map $[(x,t)] \longmapsto [(f(x),t)]\Mdot$
\end{definition}
\begin{definition}\emph{Topological $K$-theory} is the sequence of functors $\fdec{K_n}{\cat{KHaus}^\op}{\cat{Ab}}$  defined by $K_n = K \circ S^{|n|} \Mdot$  
\end{definition}
 
Bott periodicity \cite{atiyah1964periodicity} provides natural isomorphisms $K_n \simeq K_{n+2}\Mdot$  We are left with $K_0 = K$ and $K_1 = K \circ S \,$.  Note that topological $K$-theory additionally possesses a ring structure which does not survive in the noncommutative case.

\subsection{Operator $K$-theory}

Here, we outline the generalisation of topological $K$-theory to operator $K$-theory by the canonical method of noncommutative geometry.  We provide the definition and properties of the operator $K_0$-functor which we will use in our analysis of the extension of the topological $K$-functor.  These are basic facts found in any introduction to the subject, e.g. \cite{rordam,wegge1993k,fillmore}.
We start by defining $K_0$ for unital \cstar s, and then extend it to the nonunital case via unitalisation.

\subsubsection{Operator $K$-theory for unital $C^*$-algebras}
In order to generalise a topological concept to the noncommutative case, one must begin with a characterisation in terms of commutative algebra of the topological concept in question.  In the case of $K$-theory, this requires phrasing the notion of \emph{a complex vector bundle over $X$} in terms of the algebra $C(X)$ of continuous, complex-valued functions on $X$.
This rephrasing is provided by the Serre--Swan theorem:

\begin{theorem}[Serre--Swan, \cite{swan}]
The category of complex vector bundles over a compact Hausdorff space $X$ is equivalent to the category of finitely generated projective $C(X)$-modules.
\end{theorem}

Recall that a projective $\ca$-module is the direct summand of a free $\ca$-module.  
 Roughly, the module associated to a vector bundle $E$ over $X$ is the set of continuous global sections of $E$ with the obvious operations.  This justifies considering a finitely generated projective (left) $\ca$-module to represent a complex vector bundle over the noncommutative space underlying the \cstar\ $\ca$. 

The canonical translation process of noncommutative geometry suggests, having now in our possession an algebraic characterisation in terms of $C(X)$ of the topological notion of complex vector bundle, that we use it to define its noncommutative generalisation.  That is, define the unital Murray--von Neumann semigroup of a \cstar\ to be the abelian monoid of its finitely generated projective modules (up to the appropriate notion of equivalence and with an appropriate addition operation). It turns out to be more convenient to work with an algebraic gadget which is equivalent to finitely generated projective $\ca$-modules: namely, projections in a matrix algebra $M_n(\ca)$ over $\ca$.
If $\mu$ is such a finitely generated projective module $\mu$, 
then there exists another module  $\mu^\perp$ such that  $\mu \oplus \mu^\perp \simeq \ca^n$.
We thus identify the module $\mu$ with the projection $\fdec{p}{\ca^n}{\mu}$, or rather, with the canonical representation of that projection as an element of the matrix algebra $M_n(\ca)$.

Equipped with our algebraic characterisation of vector bundles, we are ready to begin defining operator $K$-theory in a manner directly analogous with the construction of topological $K$-theory.

\begin{definition}
Let $\ca$ be a \cstar.
Two projections $p \in M_n(\ca)$ and $q \in M_m(\ca)$ are Murray--von Neumann equivalent, denoted $p \simMvN q$, whenever there is a partial isometry $v$ in the \cstar\ $M_{m,n}(\ca)$ of $m \times n$  matrices over $\ca$ such that $p = vv^*$ and $q = v^*v$.
\end{definition}

\begin{definition}[The Murray--von Neumann semigroup for unital $\ca$]
Let $\ca$ be a unital \cstar.
Its \emph{Murray--von Neumann semigroup}, $V_0(\ca)$, is the set of Murray--von Neumann equivalence classes of projections in matrices over $\ca$:
\[ \left.\bigsqcup_{n \in \NN} \setdef{p \in M_n(\ca)}{p \text{ is a projection}} \middle\slash {\simMvN}\right. \Mdot \]
It is equipped with the abelian addition operation
\[[p] + [q] = \left[\left( \begin{array}{ccc}
p & 0 \\
0 & q \\
\end{array} \right) \right]\Mcomma\]
for which the equivalence class of the zero projection is a neutral element, therefore forming an abelian monoid.
A unital $*$-homomorphism $\fdec{\phi}{\ca}{\cb}$ yields a monoid homomorphism
$\fdec{V_0(\phi)}{V_0(\ca)}{V_0(\cb)}$ given by
$[p] \longmapsto [M_n(\phi)(p)]$
for each $n \in \NN$ and $p$ a projection in $M_n(\ca)$,
where $M_n(\phi)$ acts on elements of $M_n(\ca)$ by entrywise application of $\phi$.
This defines a functor $\fdec{V_0}{\cat{uC^*}}{\cat{AbMon}}$.
\end{definition}

\begin{definition}
The \emph{operator $K_{0}$-functor for unital \cstar s}, $\fdec{K_{0}}{\cat{uC^*}}{\cat{Ab}}$,
is $\cg \circ V_0$.
\end{definition}

\subsubsection{Operator $K$-theory for (nonunital) $C^*$-algebras}

So far, we have defined operator $K$-theory only for the unital case.  We describe the extension of $K_0$ to all \cstar s. 

\begin{definition}\label{def:minunitalisation}
The \emph{minimal unitalisation} of a \cstar\ $\ca$ (which itself may be unital or nonunital), is defined as the unital \cstar\ $\ca^+$ with underlying set $\ca \times \C$, componentwise addition and scalar multiplication, and multiplication and involution given by
$$(a, z)(a',z') = (aa' +z'a + za', zz') \text{, \,\,\,\,\,\,\,\,\,\,\, }(a,z)^* = (a^*, \bar z) \Mdot$$ 
There exists a unique $C^*$-norm on $\ca^+$, whose definition we omit, which extends the norm on $\ca$.
\end{definition}

  Note  that $(-)^+$ is a functor from the category $\cat{C^*}$ of \cstar s with $*$-homomorphisms to the category $\cat{uC^*}$ of unital \cstar s and unital $*$-homomorphisms: a $*$-homomorphism $\fdec{\phi}{\ca}{\cb}$ yields $(a,z) \longmapsto (\phi(a),z)$.  

A copy of $\ca$ lives canonically inside $\ca^+$ in the first component.
  Indeed, the unitalisation of a \cstar\ yields a short exact sequence
 $$0 \longrightarrow \ca \xlongrightarrow{\iota} \ca^+ \xlongrightarrow{\pi} \C \longrightarrow 0$$
 with $\iota$ being the injection into the first component and $\pi$ being the projection to the second component.
 Exactness justifies identifying $\ca$ with $\ker\pi$.

\begin{definition}
The $K_0$ group of a \cstar\ $\ca$ is the subgroup of $K_0(\ca^+)$ given by the kernel of $K_0(\pi)$.
A $*$-homomorphism $\fdec{\phi}{\ca}{\cb}$ yields a homomorphism from $\ker K_0(\ca^+ \xlongrightarrow{\pi} \C)$ to $\ker K_0(\cb^+ \xlongrightarrow{\pi} \C)$ by restriction of $K_0(\phi^+)$ to the kernel of $K_0(\ca^+\ \xlongrightarrow{\pi} \C)$.
This defines the \emph{operator $K_0$-functor}, $\fdec{K_{0}}{\cat{C^*}}{\cat{Ab}}$.
\end{definition}

\subsubsection{Higher operator $K$-groups}

As in the topological case, one can easily construct the full sequence of functors $K_n$ for $n \in \NN$.

\begin{definition}
The \emph{suspension functor} $S$ maps the category of \cstar s to itself by sending a \cstar\ $\ca $ to a sub-\cstar\ of $C(\mathbb{T}, \ca)$, the \cstar\ of continuous $\ca$-valued functions on the complex unit circle $\mathbb{T}$; $S(\ca)$ consists of those functions $\fdec{f}{\mathbb{T}}{\ca}$ such that $f(1) =  0$ (or alternatively, $S(\ca) = \ca \otimes C_0(\mathbb{R})$, the tensor product of $\ca$ with the continuous, complex-valued functions on $\mathbb{R}$ vanishing at infinity). 
A $*$-homomorphism $\fdec{\phi}{\ca}{\cb}$ is mapped to
the $*$-homomorphism $\fdec{S(\phi)}{S(\ca)}{S(\cb)}$ defined by $([S\phi](f))(t) \longmapsto \phi(f(t))$ for all $t \in \mathbb{T}$.
\end{definition}

Note that this suspension functor is an extension of the one defined on topological spaces in the sense that for algebras $\ca = C(X)$, we have that $S(\ca) = C(S(X))$.
\begin{definition}\emph{Operator $K$-theory} is the sequence of functors $\fdec{K_n}{\cat{C^*}}{\cat{Ab}}$  defined by $K_n = K_{0} \circ S^{|n|}\Mdot  $  
\end{definition}

Generalised Bott periodicity provides natural isomorphisms $K_n \simeq K_{n+2}\Mdot$  We are left with $K_0  $ and $K_1 = K_{0} \circ S$.

\subsubsection{Stability}

\begin{definition}  The \emph{compact operators $\ck$} is the sub-\cstar\ of $\cb(\ch)$, with $\ch$ a Hilbert space of countable dimension, which is generated by the finite rank operators.  

Alternatively, it is defined as the colimit (direct limit) in the category of $C^*$-\cstar s of the sequence of matrix algebras
$$M_1(\C) \hookrightarrow M_2(\C) \hookrightarrow M_3(\C) \hookrightarrow \cdots$$
where the injections are inclusion into the upper left corner: $x \longmapsto \left( \begin{array}{ccc}
x & 0 \\
0 & 0 \\
\end{array} \right)$.

\end{definition}

The \cstar\ $\ck$ is \emph{nuclear}, which means that, for any \cstar\ $\ca$, there is a unique $C^*$-norm on the algebraic tensor product $\ca \otimes_{\mathsf{alg}} \ck$ and thus we may speak unambiguously of the \cstar\ $\ca \otimes \ck$.

\begin{definition}
The \emph{stabilisation functor} from the category of \cstar s to itself,
which we (by harmless abuse of notation) denote by $\fdec{\ck}{\cat{C^*}}{\cat{C^*}}$,
maps a \cstar\ $\ca$ to the \cstar\ $\ck(\ca) = \ca \otimes \ck$.
A $*$-homomorphism $\fdec{\phi}{\ca}{\cb}$ is mapped to $\fdec{\ck(\phi)}{\ck(\ca)}{\ck(\cb)}$ defined by $\ck(\phi) = \phi \otimes id_\ck$.
 
It is alternatively defined as the direct limit of matrix algebras.
That is, $\ck(\ca)$ is the colimit in the category of \cstar s of
$$M_1(\ca) \hookrightarrow M_2(\ca) \hookrightarrow M_3(\ca) \hookrightarrow \cdots$$
where the morphisms are inclusion into the upper left corner.  A $*$-homomorphism $\fdec{\phi}{\ca}{\cb}$ yields homomorphisms $\fdec{M_n(\phi)}{M_n(\ca)}{M_n(\cb)}$ which form the components of the natural transformation yielding $\ck(\phi)$.
\end{definition}

Since the \cstar\ $\ck \otimes \ck$ is isomorphic to $\ck$,
the stabilisation functor is an idempotent operation, i.e. $\ck \circ \ck \simeq \ck$.

\begin{definition}
A \cstar\ $\ca$ is called \emph{stable} or \emph{a stabilisation} if it is fixed (up to isomorphism) by the $\ck$ functor, i.e. $\ca \simeq \ca \otimes \ck$.
\end{definition}

Note that no stable \cstar\ can be unital.
Two \cstar s $\ca$ and $\cb$ are \emph{stably equivalent} when $\ck(\ca) \simeq \ck(\cb)$.
Among stable \cstar s, stable equivalence reduces to ordinary isomorphism equivalence.
As we shall see, operator $K$-theory doesn't distinguish between stably equivalent algebras.

\begin{theorem}
Operator $K$-theory is matrix stable.  That is, $K_0 \simeq K_0 \circ M_n$ and $K_1 \simeq K_1 \circ M_n \,$,
where $M_n$ is the functor that forms $(n \times n)$-matrix algebras over \cstar s.
\end{theorem}

\begin{theorem}
Operator $K$-theory is continuous.  That is, if $$\ca_1 \longrightarrow \ca_2 \longrightarrow \ca_3 \longrightarrow \cdots$$ is a direct sequence of \cstar s and $*$-homomorphisms, $$K_0(\ca_{1}) \longrightarrow K_{0}(\ca_2) \longrightarrow K_0(\ca_3) \longrightarrow \cdots$$ is its image under the $K_0$-functor, and $\ca = \colim \,\ca_n$, then,
$$K_0(\ca) \simeq \colim \, K_0(\ca_n)$$
via the obvious homomorphism induced between cones.  A similar statement holds for $K_1$.
\end{theorem}

As a consequence of the preceding two theorems, and the alternative definition of the compact operators as the limit of a direct sequence of matrix algebras, we obtain: 
\begin{theorem}
Operator $K$-theory is stable.  That is, $K_0 \simeq K_0 \circ \ck$ and $K_1 \simeq K_1 \circ \ck$.
\end{theorem}
 
Consequently, the operator $K$-theory functors are determined by their restrictions to stable \cstar s.

\subsection{Alternative definition of  operator $K_0$-functor}
For a unital \cstar\ $\ca$, the Murray--von Neumann semigroup, and thus the $K_0$-group, can be expressed in a rather simple fashion in terms of projections of its stabilisation~\cite[Exercise 6.6]{rordam}.
We require this definition in the proof of Theorem~\ref{thm:Ktheory} and thus describe it in explicit detail.

\begin{definition}
Two projections $p$ and $q$ in a \cstar\ $\ca$ are \emph{unitarily equivalent}, denoted by $p \simu q$, whenever there is a unitary $u \in \ca^+$ such that $p = uqu^*$.
We write $[p]$ for the unitary equivalence class of $p$.
\end{definition}

Given projections $p_1, \ldots, p_k \in \ca \otimes \ck$,
one can find pairwise orthogonal representatives of their unitary equivalence classes,
i.e. there exist projections $q_1, \ldots q_k \in \ca \otimes \ck$ such that $p_i \simu q_i$ ($i \in \{1, \ldots,n\})$ and all the $q_i$ are pairwise orthogonal \cite[Exercise 6.6]{rordam}. 

The Murray--von Neumann semigroup for unital \cstar s admits the following alternative characterisation:
\begin{definition}[The Murray--von Neumann semigroup for unital $\ca$, alternative definition]
Let $\ca$ be a unital \cstar.
The elements of $V_0(\ca)$ are the unitary equivalence classes of projections in $\ck(\ca)$.
The abelian addition operation is given by orthogonal addition. That is, if $p$ and $p'$ are two  projections in $\ca \otimes \ck$, then
$[p] + [p'] = [q + q']$
where $q$ and $q'$ are orthogonal representatives of $[p]$ and $[p']$, respectively (i.e. $p \simu q$, $p' \simu q'$, and $q \perp q'$).
The equivalence class of the zero projection is a neutral element for this operation, making $V_0(\ca)$ an abelian monoid.
A unital $*$-homomorphism $\fdec{\phi}{\ca}{\cb}$ yields a monoid homomorphism by  $\fdec{V_0(\phi)}{V_0(\ca)}{V_0(\cb)}$ by $[p] \longmapsto [\ck(\phi)(P)]$.
This defines a functor $\fdec{V_0}{\cat{uC^*}}{\cat{AbMon}}$.
\end{definition}

Through this reformulation of the Murray--von Neumann semigroup functor $V_0$, we automatically get a new description of $K_0$ by composition with the Grothendieck group functor, as $K_0 = \cg \circ V_0$.  Then, $K_0(\ca)$ is simply the collection of formal differences $$[p] - [q]$$ of elements of $V_0(\ca)$ with $$[p] - [q] = [p'] - [q']$$ precisely when there exists $[r]$ such that $$[p] + [q'] + [r] = [p'] + [q] + [r]\Mdot$$

Composing the action on morphisms of the Grothendieck group functor after the action of $V_0$ just defined, we find that a unital $*$-homomorphism $\fdec{\phi}{\ca}{\cb}$ between unital \cstar s yields an abelian group homomorphism between the $K_0$ groups of $\ca$ and $\cb$ given by 
$$[p] - [q] \longmapsto [\phi(p)] - [\phi(q)] \Mdot$$

\section{Ideals of operator algebras}\label{app:ideals}

\subsection{The primitive ideal space}

Here, we include some basic facts on the prime ideal spectrum of rings and on its $C^*$-algebraic analogue, the primitive ideal space.
These are required for our explication of the motivation for considering the extension of the closed-set lattice functor.

\subsubsection{The spectrum of commutative rings}

In commutative ring theory and algebraic geometry, the starting point for the application of geometrical methods is the association of topological spaces to rings \cite[p. 70]{hartshorne}. These are, in fact, locally ringed spaces; however, we will not be considering this additional structure.

\begin{definition}
A \emph{prime ideal} $J$  of a commutative ring $R$ is a ideal $J \subsetneq R$ such that whenever we have $a, b \in R$ such that $ab \in J$ then either $a \in J$ or $b \in J$.
\end{definition}
The canonical examples of prime ideals come from the ideals of the ring of integers generated by prime numbers.

\begin{definition}
Let $R$ be a commutative ring and let $I \subset R$ be a two-sided ideal of $R$.  Then $\mathrm{hull}(I)$ is the set of prime ideals that contain $I$.
\end{definition}

\begin{definition}
Let $R$ be a commutative ring.  The contravariant \emph{spectrum functor} $\fdec{\Spec}{\cat{Rng}}{\cat{Top}}$ from the category of rings and ring homomorphisms to the category of topological spaces and continuous maps is defined as follows.

Given a ring $R$, $\Spec(R)$ is the set of prime ideals of $R$, equipped with the \emph{hull-kernel (or Zariski, or Jacobson) topology}, whose closed sets are of the form $\mathrm{hull}(I)$ for some two-sided ideal $I \subset R$.

The $\Spec$ functor sends a ring morphism $\fdec{h}{R}{S}$ to the continuous map $\fdec{\Spec(h)}{\Spec(S)}{\Spec(R)}$  that maps a prime ideal $J$ to its preimage $h^{-1}(J)$ under $h$.
\end{definition}


\subsubsection{The primitive ideal space}
These definitions and theorems can be found in \cite[p. 208]{AlfsenShultz1} and \cite[p. 118]{Blackadar}.

\begin{definition}
A \emph{primitive ideal} $J$ of a \cstar\ $\ca$ is an ideal that is the kernel of an irreducible representation of $\ca$.
\end{definition}

Recall that an irreducible representation of a \cstar\ $\ca$ is a $*$-representation $\fdec{\pi}{\ca}{\cb(\ch)}$ such that no nontrivial closed subspaces $S \subset \ch$ satisfy $\pi(a)S \subset S$ for all $a \in \ca$.
Every pure state of $\ca$ gives rise to an irreducible representation $\ca$ by the \Gelfand--\Naimark--Segal construction.

\begin{definition}
Let $\ca$ be a \cstar\ and let $I \subset \ca$ be a closed, two-sided ideal of $\ca$.  Then $\mathrm{hull}(I)$ is the set of primitive ideals containing $I$.
\end{definition}

\begin{definition}
Let $\ca$ be a \cstar.  The \emph{primitive ideal space} $\Prim(\ca)$ is the set of the primitive ideals of $\ca$, equipped with the \emph{hull-kernel (or Zariski, or Jacobson) topology} whose closed sets are of the form $\mathrm{hull}(I)$ for some closed, two-sided ideal $I \subset \ca$.
\end{definition}


\begin{theorem}The map $\mathrm{hull}$ is an order preserving bijection between the set of two-sided ideals of a \cstar\ $\ca$ and the closed sets of the hull-kernel topology on $\Prim(\ca)$.
\end{theorem}

\begin{definition}The \emph{spectrum} $\hat \ca$ of a \cstar\ $\ca$ is the set of unitary equivalence classes of irreducible representations of $\ca$.
It is equipped with the coarsest topology with respect to which the map $[\pi] \longmapsto \ker\pi$ is continuous.

\end{definition}

The topology on $\hat \ca$ is thus also order isomorphic to the partially ordered set of two-sided ideals of $\ca$.

\subsection{Von Neumann algebras}

We briefly outline some required elementary facts about von Neumann algebras \cite[Chapter 3]{AlfsenShultz1}.

\begin{definition}  A \emph{von Neumann algebra} $\ca$ is a $*$-subalgebra of $\cb(\ch)$, for some Hilbert space $\ch$, which is closed in the weak (operator) topology.
\end{definition}

Recall that a net of operators $(T_\alpha)$ in $\cb(\ch)$ converges to $T$ in the \emph{weak  topology} if and only if, for every vector $v \in \ch$ and linear functional $\phi \in \ch^*$, we have that  $(\phi(T_\alpha(v)))$ converges to $\phi(T(v))$.  As convergence of a net of operators in norm  implies its weak convergence, we see that von Neumann algebras are examples of \cstar s.  We may equally well have defined von Neumann algebras to be $*$-subalgebras of $\cb(\ch)$ which are closed in the strong, ultraweak, or ultrastrong topologies as the closures of $*$-subalgebras of $\cb(\ch)$ in these topologies all coincide.  Von Neumann proved that taking any of these closures of unital $*$-subalgebras of $\cb(\ch)$ coincides also with taking the double commutant (though he did not know of the ultrastrong topology). 

We will primarily require facts about projections and ideals of von Neumann algebras and the relationship between the two notions.

\subsubsection{Projections}
The \emph{projections} of $\ca$ are operators $p$ such that $p = p^* = p^2$. They are orthogonal projections onto closed subspaces of $\ch$.  This yields a natural partial order on projections induced  by the inclusion relation on their corresponding subspaces.  Alternatively, this order can be defined by:
$$ p \leq q \Miff pq = p \Miff qp = p \Mdot$$

  We denote the partially ordered set of projections in $\ca$ by $\cp(\ca)$.  In von Neumann algebras, the collection of projections forms a complete lattice: the infimum $\inf_\alpha p_\alpha$ of an arbitrary collection of projections $\{p_\alpha\}_\alpha$ is given by the orthogonal projection onto $\bigcap_\alpha p_\alpha \ch$ whereas the supremum $\sup_\alpha p_\alpha$ is the orthogonal projection onto the closed linear span of $\bigcup_\alpha p_\alpha \ch$.
 The orthogonal complement map $(-)^\perp$ that sends $p$ to $1-p$ makes this lattice complemented in the sense that $p \vee p^\perp = 1$, $p \wedge p^\perp = 0$, and $p^{\perp \perp} = p$.

The set of projections in $\ca$ is also equipped with several other preorders which arise from the canonical partial order and certain compatible equivalence relations.  We will require, in particular, the notions of Murray--von Neumann equivalence of projections and unitary equivalence of projections.

The intuition behind Murray--von Neumann equivalence is to identify projections whose corresponding image subspaces are of the same dimension.  That is, there should be an operator $v \in \ca$ mapping the Hilbert space $\ch$ to itself which isometrically maps the subspace of one projection to the subspace of another, thereby witnessing the equality of their dimension.

\begin{definition}Two projections $p$ and $q$ in a von Neumann algebra $\ca$ are \emph{Murray--von Neumann equivalent}, denoted $p \simMvN q$, if and only if there exists $v \in \ca$ such that $$p = v^*v \text{ and } q=vv^*\Mdot  $$\end{definition}

The partial order on $\cp(\ca)$ induces a partial order on the set of Murray--von Neumann equivalence classes of projections.  We write $p \leqMvN q$ to denote that $p \simMvN q'$ for some $q' \leq q$.

\begin{definition}Two projections $p$ and $q$ in a von Neumann algebra $\ca$ are \emph{unitarily equivalent}, denoted $p \simu q$, if and only if there exists  a unitary element $u \in \ca$ such that $p = uqu^*$.\end{definition}

Similarly, the partial order on $\cp(\ca)$ induces a partial order on the set of unitary equivalence classes of projections.  We write $p \lequ q$ to denote that $p \simu q'$ for some $q' \leq q$.

Unitary equivalence (resp. ordering)
implies Murray--von Neumann equivalence (resp. ordering) for arbitrary pairs of projections.
We will require the following partial converse for orthogonal projections.
\begin{proposition}\label{prop:mvNimpliesunitary}
Let $p$ and $q$ be projections in a von Neumann algebra.
If $p$ and $q$ are orthogonal, then $p \simMvN q$ iff $p \simu q$, and, moreover, $p \leqMvN q$ iff $p \lequ q$.
\end{proposition}

For a proof of the statement concerning equivalences, see \cite[Proposition 6.38]{AlfsenShultz1}.
The second statement is an easy consequence of this:
$p \leqMvN q$ means that $p \simMvN q'$ for some $q' \leq q$;
but if $p$ and $q$ are orthogonal then so are $p$ and $q'$;
and so by the first statement one obtains $p \simu q'$, meaning that  $p \lequ q$.

\begin{definition}
The central carrier $C(p)$ of a projection $p \in \ca$ is the smallest central projection above $p$:
$$C(p) = \inf \setdef{z \in \cp(\ca)\cap \cz(\ca)}{p \leq z}\Mdot$$ \end{definition}
It is immediate from this definition that a projection $p$ and a unitary rotation $upu^*$ have the same central carrier for $p \leq z$ if and only if $upu^* \leq uzu^* = zuu^* = z$.  It is also immediate that if $S \subset \cp(\ca)$ is a set of projections, then $C(\sup S) = \sup_{p \in S}C(p)$.

One of the basic technical tools we will require is the comparison theorem of projections in a von Neumann algebra \cite{KadisonRingrose1}.  The intuitive idea is best understood in a factor (a von Neumann algebra with trivial centre) which can be thought of as an elementary direct summand.  Here, the dimension of two projections can be compared; either they are of equal dimension, or the dimension of one exceeds the dimension of the other.     

\begin{theorem}[Comparison theorem] \label{comparison}
Let $p$ and $q$ be projections in a von Neumann algebra $\ca$.  There exists a central projection $z$ in $\ca$ such that \[zp \geqMvN  zq \;\;\text{ and }\;\; z^\perp p \leqMvN z^\perp q \Mdot\]
\end{theorem}
 
\subsubsection{Ideals}

Ideals of operator algebras must satisfy both the usual algebraic conditions as well as an additional topological condition.  

It turns out that the appropriate notion of morphism for von Neumann algebras is not weakly continuous $*$-homomorphism but rather ultraweakly continuous $*$-homomorphism. 
 The ultraweak topology is stronger than the weak topology.

\begin{definition} \label{idealdef}
A \emph{left (resp. right) ideal} $I$ of a von Neumann algebra $\ca$ is a left (resp. right) ring ideal $I \subset \ca$ that is closed in the ultraweak topology.

A \emph{total ideal} or \emph{two-sided ideal} $I$ of a von Neumann algebra $\ca$ is two-sided ring ideal $I \subset \ca$ that is closed in the ultraweak topology.
\end{definition}

Left, right, and total ideals correspond with projections.  Examples of left (resp. right) ideals are the sets given by $\ca p$ (resp. $p \ca$). These are the kernels of morphisms given by right (resp. left) multiplication by $p^\perp$.

\begin{theorem}[{\cite[Proposition 3.12]{Takesaki1979:TheoryOfOperatorAlgebrasI},\cite[Lecture 9, Corollary 6]{Lurie2011:261y}}]
Every left ideal $L \subset \ca$ of a von Neumann algebra $\ca$ is of the form $L = \ca p$ for a projection $p \in \cp(\ca)$.  Further, the projection $p$ is uniquely determined by $L$.
\end{theorem}

Under this correspondence, total ideals are precisely those left or right ideals corresponding to central projections.

\begin{theorem}[{\cite[Proposition 3.12]{Takesaki1979:TheoryOfOperatorAlgebrasI}, \cite[Lecture 9, Corollary 8]{Lurie2011:261y}}]\label{idealcentral}
Every total ideal $I \subset \ca$ of a von Neumann algebra $\ca$ is of the form $I = z\ca z = z\ca = \ca z$ for a unique central projection $z \in \cp(\ca)\cap \cz(\ca)$.
\end{theorem}

\bibliographystyle{amsplain}
\bibliography{thesispaper_CMP} 

\end{document}